\documentclass[a4paper,11pt]{amsart}
\usepackage{graphicx}
\usepackage{amsmath}
\usepackage{amsfonts}
\usepackage{tikz-cd}
\usepackage{mathrsfs}
\usepackage{multirow}
\usepackage{multicol}
\usepackage{float}
\usepackage{caption}
\usepackage{amssymb}
\usepackage{tikz}
\usepackage{xcolor}
\usepackage{hyperref}
\usepackage{faktor}

\usepackage[
  a4paper,
  left=2.5cm,
  right=2.5cm,
  top=3cm,
  bottom=3cm
]{geometry}

\newtheoremstyle{introductionplain}
  {\topsep}     
  {\topsep}     
  {\itshape}    
  {0pt}        
  {\bfseries}  
  {.}          
  {0.5em}      
  {}           

\theoremstyle{introductionplain}
\newtheorem*{definition*}{Definition}
\newtheorem*{theorem*}{Theorem}
\newtheorem*{proposition*}{Proposition}
\newtheorem*{corollary*}{Corollary}
\theoremstyle{plain}

\usepackage{needspace}

\newtheorem{theorem}{Theorem}[section]

\usepackage{algpseudocode}

\newtheorem{corollary}[theorem]{Corollary}

\newtheorem{definition}[theorem]{Definition}
\newtheorem{example}[theorem]{Example}

\newtheorem{lemma}[theorem]{Lemma}
\newtheorem{notation}[theorem]{Notation}

\newtheorem{proposition}[theorem]{Proposition}

\newtheorem{remark}[theorem]{Remark}

\newtheorem{observation}[theorem]{Observation}

\newenvironment{definitionA}
  {\par\medskip\noindent\textbf{Definition 1.}\itshape}
  {\par\medskip}

  \newenvironment{definitionB}
  {\par\medskip\noindent\textbf{Definition 2.}\itshape}
  {\par\medskip}

\newenvironment{theoremA}
  {\par\medskip\noindent
   \textbf{Theorem A.} (Theorem \ref{thm defect})\ \itshape}
  {\par\medskip}

  \newenvironment{theoremB}
  {\par\medskip\noindent\textbf{Theorem B.} (Theorem \ref{teorema formula}) \itshape}
  {\par\medskip}

   \newenvironment{theoremC}
  {\par\medskip\noindent\textbf{Theorem C.} (Theorem \ref{teorema cuadrado}) \itshape}
  {\par\medskip}

    \newenvironment{theoremD}
  {\par\medskip\noindent\textbf{Theorem D.} (Proposition \ref{prop:cota})\itshape}
  {\par\medskip}

\newcommand\CC{{\mathcal C}}

\newcommand\Tan{\operatorname{Tan}}

\title{Degrees and Directional Defects of Embedded Algebraic Vector Bundles}

\author{Leonardo Lanciano}

\date{}

\address{ Universidad de Buenos Aires, Facultad de Ciencias Exactas y Naturales,
Departamento de Matemática, Buenos Aires, Argentina.}
\email{llanciano@dm.uba.ar}

\begin{document}
\begin{abstract}
We develop a unified defect theory for embedded algebraic vector
bundles which places several classical constructions within a common
framework. The theory recovers tangential, dual, join, and secant
defects. We first prove an effective numerical criterion characterizing defectivity by the vanishing of a bidegree. This also allows us to recover the exact defect from the vanishing pattern of the bidegrees, extending a theorem of Holme. Using van der Waerden's theorem on bidegrees, we obtain a formula that decomposes the geometric degree of an embedded vector bundle into contributions
from the direction varieties of its general linear restrictions. This
leads to a characterization of embedded algebraic vector bundles of minimal degree.

As a main application, we establish the sharp universal bound $\deg(TV)\leq\deg(V)^2$
for every smooth irreducible affine variety \(V \subseteq \mathbb{A}^n\) and $TV\subseteq \mathbb{A}^{2n}$ its tangent bundle. Finally, under suitable
regularity assumptions, we apply the quadratic bound to prolongation
varieties of differential algebraic systems, obtaining uniform degree estimates
and thereby providing a partial answer to an open problem in
differential algebra posed by Pogudin.
\end{abstract}
\maketitle
\begin{small}{\textbf{Keywords.}
Vector bundles, Defectivity, Geometric degree, Prolongation varieties, Tangential variety.}
\end{small}

\section{Introduction}
The geometry of varieties swept out by linear spaces is a classical and relevant topic in algebraic geometry. Tangential varieties, projective dual varieties, join and secant varieties are notable examples of varieties of this kind. These arise naturally in the study of projective embeddings, projections and incidence problems. In particular, the study of tangential varieties played a fundamental role in Zak's remarkable contribution on Hartshorne's conjecture on linear normality (see \cite[Corollary 3]{Zak1983} and \cite[Conjecture 4.2]{Hartshorne1974}).

The origins of these ideas can be traced back at least to the works of Terracini \cite{Terracini1911} and Severi \cite{Severi1902}. The theory was subsequently developed from several different perspectives. Secant and tangential varieties were systematically studied for instance in \cite[Chapters 1-2]{Zak1993}. On the other hand, projective dual varieties were studied using Chern classes where several formulas for their dimension and degrees were obtained by Holme \cite{Holme1979DualSmooth,Holme1988Duality} (see \cite{GKZ1994} for a modern approach). Lastly, joins were extensively studied in \cite[Chapter 4]{flenner1999joins}. In recent years, these classical constructions have received renewed attention, with much of the literature focusing on the explicit computations of dimensions and degrees of tangential or secant varieties in particular families such as Segre-Veronese varieties (see for instance \cite{HernandezGomezRusso2026,ChiantiniCiliberto2002,ChiantiniCiliberto2010,AboVannieuwenhoven2018,BernardiEtAl2018,ChiantiniCoppens2001,Russo2016} and the references in there).

Across all these settings, one of the central questions is whether the relevant family of linear spaces (for instance tangent spaces or secant lines) sweeps out a variety having the dimension predicted by the natural parameter count. When the answer to that question is negative, the notion of \emph{defectivity} arises. Despite the extensive literature devoted to the individual constructions mentioned above, defectivity itself has rarely been studied from such a general point of view.

We start by introducing a general notion of defectivity following \cite{MezzettiTommasi2004,DePoiMezzetti2007,ChiantiniCiliberto2010} and \cite[Section 4.2.3]{EisenbudHarris2016}.

\medskip

\begin{definitionA}
   Let $X \subseteq \mathbb{P}^m$ be an irreducible projective variety.
    \begin{itemize}
        \item  We say $X$ is \emph{swept out by linear spaces} if there exist integers $n \in \mathbb{N}$, $k \in \mathbb{N}_0$, an irreducible quasi-projective variety $B \subseteq \mathbb{P}^n$ and a regular map $\gamma_X: B \rightarrow \mathbb{G}(m+1,k+1)$ such that: $$X = \overline{\left \{ [v] \in \mathbb{P}^m \mid \exists \ p \in B , v \in \gamma_X(p) \setminus \{0\}\right \}} = \overline{\bigcup_{p \in B} \mathbb{P}(\gamma_X(p))},$$ where the bars denote the Zariski closure and $\mathbb{G}(m,r)$ denotes the Grassmannian of $r$-dimensional vector subspaces of $\mathbb A^m$
        \item  We define the defect of \((X,\gamma_X)\), by
$$
\delta(X,\gamma_X):=
\min\{\dim B+k,m\}-\dim X.
$$
    \end{itemize}

\end{definitionA}

Most available techniques for detecting defectivity are restricted to a specific geometric situation and exploit the particular structure of the variety under consideration. For instance, in the setting of dual varieties, defectivity can be detected through the vanishing of suitable Chern classes (see, for example, \cite{Holme1988Duality,GKZ1994}) or polar degrees (see \cite{https://doi.org/10.1112/blms.12379,KOHN2021157,HOLME2001363}). Nevertheless, to our knowledge, there is no general systematic framework that turns the defectivity problem for an arbitrary family of linear spaces into an effective computation.

\medskip

The first aim of this paper is to develop a common numerical
framework for these defectivity problems through embedded
algebraic vector bundles. We begin by introducing the basic objects and invariants motivated by \cite{MezzettiTommasi2002,DePoiMezzetti2007}, \cite[II, Exercise 5.18]{hartshorne1977algebraic} and \cite[Remark 3.4]{https://doi.org/10.1112/blms.12379}. 
 
\medskip

\begin{definitionB}
Let \(V\subseteq\mathbb A^n\) be an irreducible algebraic variety of dimension \(d\), and let
\(E\subseteq V\times\mathbb A^m\) be an algebraic variety.

\begin{enumerate}
\item We say that \(E\) is an {embedded algebraic vector bundle of rank \(r\) over \(V\)} if the canonical first projection $\pi_1:E \rightarrow V$ is Zariski locally trivial and the fibers are of the form $\pi_1^{-1}(p):=\{p\} \times L_p$ where $L_p$ is an \(r\)-dimensional linear subspace of \(\mathbb A^m\) for every $p \in V$.

\item The {variety of directions} of \(E\), $\mathscr{D}(E) \subseteq \mathbb{A}^m,$ is the variety swept out by its fibers:
\[
\mathscr{D}(E):=\overline{\bigcup_{p\in V}L_p}\subseteq\mathbb A^m.
\]
\item The Grassmann classifying map of $E$ is
\[
 \Gamma_E:V\longrightarrow\mathbb{G}(m,r),
 \qquad p\longmapsto L_p,
\]
\item The {directional defect} of \(E\) and the {Grassmann} defect of $E$ are:
\[
\Delta(E):=\min\{d+r,m\}-\dim\mathscr{D}(E), \qquad \Delta_G(E):=d-\dim\overline{\Gamma_E(V)} .
\]

\end{enumerate}
\end{definitionB}

\medskip

The geometric meaning of these two defects is complementary. The directional defect measures the failure of the fibers of \(E\) to sweep out a variety of the expected dimension, while the Grassmann defect measures the failure of the fibers as points in the Grassmannian to vary with the expected number of parameters. The Grassmann classifying map is a natural generalization of the
Gauss map to vector bundles and has been studied from different perspectives in the literature (see \cite{MezzettiTommasi2004,Ran1984,Ran2024}).

On the other hand, the varieties of directions $\mathscr{D}(E)$ are the affine analogues from the vector bundle point of view, of varieties $X \subseteq \mathbb{P}^n$ swept out by linear spaces. In particular, it is easy to see that the defect theories from Definitions $1$ and $2$ agree in the sense that there is a correspondence between varieties swept out by linear spaces and embedded algebraic vector bundles that sends $\delta(X,\gamma_X)$ to $\Delta(E)$ and $\gamma_X$ to $\Gamma_E$ (see Proposition \ref{defectia}). This leads to the following natural question:
$$\textit{What is gained by translating the problem of defectivity into this language?}$$

The geometry behind varieties swept out by linear spaces can be
viewed as a projection problem, where the incidence information
is often forgotten. From the point of view of equations, this is exactly an
\emph{elimination} problem, which is algebraically difficult even
in specific examples (see \cite{LandsbergOttaviani2013,OedingRaicu2014}) and can also be computationally very expensive. The vector bundle point of view preserves the incidence structure and when restricted to the classical examples (e.g. secants, tangents and duals) the corresponding vector bundle has \emph{concrete} equations we can easily obtain from the equations of $V$ and hence, avoids this elimination step. 

By this means, the problem of defectivity can be translated into an effective numerical criterion (see Theorem A below). In the conormal case, Theorem~A recovers the polar-degree
vanishing criterion of Holme recalled in
\cite[Theorem~3.3]{https://doi.org/10.1112/blms.12379}.
Jorgenson uses this criterion, together with additional
numerical constraints provided by Huh (see \cite[Theorem 21]{Huh2012}), to establish several cases of the duality
defect conjecture posed through explicit computations.
Our framework provides a common setting for developing analogous
approaches to other defectivity problems using techniques from
polynomial equation solving (see Remark \ref{remark approach}). These include the
Abo-Ottaviani-Peterson conjecture on secant varieties of Segre
varieties (see \cite{AboOttavianiPeterson2009}) and the
Baur-Draisma-de Graaf conjecture on secant varieties of
Grassmannians (see \cite{BaurDraismaDeGraaf2007}).

\subsection{Main Results} \ 

\medskip

The paper contains three types of results. The first consists of general theorems on embedded algebraic vector bundles, their geometric degree, and their defects. The second concerns applications of these general results to classical constructions. In particular, we obtain new optimal degree bounds for the geometric degree of the tangent bundle of a smooth algebraic variety, giving a positive answer to a previously posed conjecture (see \cite[Question 33]{jeronimo2025geometric}). Finally, these estimates are applied to give a partial answer to a conjecture posed by Pogudin concerning the degree of the prolongation variety (see \cite{PogudinOpenProblems2025}).

\subsubsection{Degree formulas and numerical criteria} 
\ \medskip

We first obtain an effective numerical criterion for directional defectivity applying \cite[Theorem A]{CASTILLO2020107382} to the biprojective interpretation (see Definition \ref{def X'}) of the projective closure of $E$:

\medskip

\begin{theoremA}\
Let $V\subseteq\mathbb A^n$ be an irreducible algebraic variety of dimension $d$, and let $E\subseteq V\times\mathbb A^m$ be an embedded algebraic vector bundle of rank $r>0$. Let $\overline{E}'$ be the biprojective interpretation of its projective closure, and set $c:=\max\{d+r-m,0\}.$ Then
\[
\Delta(E)>0
\quad\Longleftrightarrow\quad
\deg_{c,d+r-1-c}(\overline{E}')=0,
\]
where $\deg_{c,d+r-1-c}(\overline{E}')$ denotes the bidegree of the biprojective variety $\overline{E}'$ (see Definitions \ref{def X'} and \ref{def:bidegree}).
\end{theoremA}

Hence, this theorem reduces the problem of defectivity to the vanishing of a single bidegree. Moreover, it follows immediately that the defect itself can be recovered from the number of consecutive vanishing bidegrees (see Corollary~\ref{cor:defect-bidegrees}). In particular, when $E$ is the conormal bundle (see Example \ref{ejemplo fibrados}), this recovers a classical Theorem from Holme (see \cite[Theorem~1.1]{HOLME2001363}), which characterizes the duality defect in terms of the associated polar degrees.

\medskip

Our second main contribution is a degree formula for arbitrary embedded algebraic vector bundles, extending \cite[Theorem 21]{jeronimo2025geometric} from tangent bundles of smooth curves to embedded algebraic vector bundles of arbitrary rank over varieties of arbitrary dimension. Van der Waerden's theorem on bidegrees (see \cite{VANDERWAERDEN1978303}) expresses the geometric degree of a projective variety defined by bihomogeneous equations as the sum of the bidegrees of its biprojective interpretation. We apply this theorem to the projective closure of an embedded algebraic vector bundle and give a geometric meaning to each summand in terms of direction varieties of restricted bundles $E_a$ and associated multiplicities $\xi_{d-a+1}(E)$ by generalizing \cite[Lemma 4.3]{CaminataCidRuizConca2023} (see Proposition \ref{producto} and Remark \ref{rmk: w(E)}). This leads to the following Theorem:

\vspace{2cm}

\begin{theoremB}
Let $V\subseteq\mathbb A^n$ be an irreducible algebraic variety of dimension $d$, and let $E\subseteq V\times\mathbb A^m$ be an embedded algebraic vector bundle of rank $r>0$. For $1\leq a\leq d$, let $E_a$ be the restriction of $E$ to a general $a$-dimensional affine linear section of $V$. Set $s:= \min{\{d,m-r \}}-\Delta(E)$, then:
\[
\deg(E)
=
\deg(V)
+
\xi_{d-s+1}(E)\deg\bigl(\mathscr D(E)\bigr)
+
\sum_{a=1}^{s-1}
\xi_{d-a+1}(E)\deg\bigl(\mathscr D(E_a)\bigr),
\]
where $\xi_{d+1}(E)=0$ and for $1 \leq a \leq s,$ $\xi_{d-a+1}(E)$ is the degree of the projection $E_a\to\mathscr{D}(E_a)$ when this map is generically finite, and is zero otherwise.  
\end{theoremB}
When nonzero, the coefficient $\xi_{d-a+1}(E)$ counts the points of the linear section whose fibers contain a general direction in $\mathscr{D}(E_a)$. When considering $E$ to be the appropriate bundle, these quantities recover familiar enumerative invariants placing them within the same framework. Examples of this are, the tangent degree $\tau(X)$ studied in \cite{HernandezGomezRusso2026} (see Remark \ref{remark seto}), the secant degree $\mu(X)$ (see Definition \ref{def:secantdegree} and Proposition \ref{prop:secant-multiplicities}) and the invariant $\omega(V)$ introduced in \cite[Definition 17]{jeronimo2025geometric} (see Subsection \ref{multiplicities}) which counts the number of points from a smooth curve which have a sufficiently general tangent line. 

Finally we show that embedded triviality is a property characterized by the minimality of the degree of $E$ (see Corollary \ref{coro: triviales}). This is also an extension of \cite[Theorem 35]{jeronimo2025geometric} into the setting of embedded algebraic vector bundles.

\subsubsection{Sharp bounds for the geometric degree of the tangent bundle} \ \medskip

Let $V\subseteq\mathbb A^n$ be a smooth irreducible affine
variety, and let $X\subseteq\mathbb P^n$ be its projective
closure. When $X$ is smooth, the bidegree formula and its
interpretation in terms of Severi's ceti
(see Proposition~\ref{proposetomultig}), together with
Severi's double point formula, yield precise degree estimates.
A deformation argument (see Lemma~\ref{Lema defomracion})
establishes the quadratic bound without assuming that $X$
is smooth. Together, these results give the following theorem.

\begin{theoremC}
Let $V\subseteq\mathbb A^n$ be a smooth irreducible affine
variety, with its tangent bundle naturally embedded as
$TV\subseteq\mathbb A^{2n}$ and let $X \subseteq \mathbb{P}^n$ be its projective closure. Then
$$\deg(TV)\leq \deg(V)^2.$$
Moreover: \begin{enumerate}
    \item For every $d,D \in \mathbb{N}$ such that $d+1 \leq n$ there exists an irreducible smooth affine variety $V_{D,d}$ with $\deg(V_{D,d})=D$ and $\dim(V_{D,d}) = d$ such that the previous inequality is an equality.
    \item If $X$ is also smooth and $2\dim(V) \leq n$, then the stronger bound holds:
    $$\deg(TV) \leq \deg(V)^2-2\mu(X)\deg(\operatorname{Sec}(X)),$$
where $\mu(X)$ is the secant degree and $\operatorname{Sec}(X)$ is the secant variety. (See Definitions \ref{def tangt} and \ref{def:secantdegree}) \end{enumerate}
\end{theoremC}
This extends the estimate for varieties defined by generic polynomials
in \cite[Corollary~32]{jeronimo2025geometric} and settles
\cite[Question 33]{jeronimo2025geometric} in full generality. Finally, general effective degree bounds for the tangential and secant varieties can be deduced from Theorem C (see Corollary \ref{cor:tangential-degree-bounds}).

\subsubsection{Applications to prolongation varieties}

\ \medskip

One of the natural motivations for studying the geometric degree of the tangent bundle comes from differential algebra (see \cite[Section~1]{jeronimo2025geometric}). Standard effective methods for deciding whether a system of differential algebraic equations \(\Sigma\) admits a solution, rely on successive differentiations, usually referred to as \emph{prolongations}, with the objective of reducing this differential problem to an algebraic one after a finite amount of steps. Motivated by the fact that the complexity of solving a polynomial system is closely related to the geometric degree of the underlying algebraic variety (see \cite{GIUSTI19971223,Castro2003TheHO}), it is therefore natural to study the degrees of prolongation varieties. This problem has been explicitly raised by Pogudin as a general open problem in differential algebra, even for polynomial dynamical systems (see \cite[\S2, Question~5]{PogudinOpenProblems2025}). The next theorem gives a partial answer to this question under reasonable hypotheses.
\begin{theoremD}
Let $\Sigma$ be a system of differential equations over an algebraically closed field of constants of characteristic $0$. Assume that its equations generate a prime ideal and that its algebraic zero set $Z(\Sigma)$ is smooth. Then
$$ \deg\bigl(Z(\operatorname{Prol}(\Sigma))\bigr)
 \leq\deg\bigl(Z(\Sigma)\bigr)^2.$$
\end{theoremD}
The estimate can be iterated whenever the same hypotheses hold at the successive stages.

\subsection*{Outline}

The paper is organized as follows. Section~2 collects the preliminary material. Section~3 introduces embedded algebraic vector bundles and their varieties of directions. Section~4 develops the general defect theory and its numerical criteria. Section~5 establishes the degree formula and studies bundles of minimal degree. Section~6 applies the theory to tangent and secant constructions and proves the quadratic bound for tangent bundle degrees. Finally, Section~7 applies these results to prolongation varieties.

\section{Preliminaries}

Throughout the paper, $K$ denotes an algebraically closed field of
characteristic zero, and all varieties and morphisms are defined over $K$. We will use $\mathbb{A}^n$ to refer to the Zariski $n$-dimensional affine space over $K$ and $\mathbb{P}^n$ for the $n$-dimensional projective space over $K$.
The term \emph{generic} means outside a proper Zariski-closed subset of the
relevant parameter space.

We write
\[
    \mathbf x=(x_1,\ldots,x_n),\qquad
    \overline{\mathbf x}=(x_0,\mathbf x)=(x_0,x_1,\ldots,x_n),
    \qquad
    \mathbf y=(y_1,\ldots,y_m).
\]

For a family of polynomials $\mathcal F\subseteq K[\mathbf x]$,
we denote its common zero locus in $\mathbb A^n$ by $Z(\mathcal F)$.
For an affine variety $V\subseteq\mathbb A^n$, we write $I(V)$
for the ideal of polynomials vanishing on $V$.
We use the same notation in the projective setting, where zero
loci are defined by homogeneous polynomials in
$K[\overline{\mathbf x}]$ and $I(X)$ denotes the homogeneous
vanishing ideal of $X\subseteq\mathbb P^n$.

\medskip

All closures are taken in the Zariski topology. The
\emph{projective closure} of an affine variety
$V\subseteq\mathbb A^n$, denoted by $\overline V$, is its closure
in $\mathbb P^n$ under the standard embedding
\[
    \mathbb A^n\hookrightarrow\mathbb P^n,
    \qquad
    (p_1,\ldots,p_n)\longmapsto[1:p_1:\cdots:p_n].
\]

We write $\mathbb G(n,k)$ for the Grassmannian of $k$-dimensional linear subspaces of $K^n$. For a nonzero subspace $W \subseteq K^n$, we denote by
$\mathbb P(W)$ its projectivization. We use the same convention for conical subsets of $\mathbb{A}^n$.

For a nonempty subset $S\subseteq\mathbb P^n$, we denote by
$\langle S\rangle$ its \emph{projective linear span}, namely
the smallest projective linear subspace containing $S$.
In particular, for distinct points $p,q\in\mathbb P^n$,
$\langle p,q\rangle$ denotes the projective line through them.
A projective variety $X\subseteq\mathbb P^n$ is called
\emph{non degenerate} if it is not contained in a hyperplane $H \subseteq \mathbb{P}^n$.

The hyperplanes of $\mathbb P^n$ themselves form a projective
space of dimension $n$, called the \emph{dual projective space}
and denoted by $(\mathbb P^n)^\vee$. Its structure
is obtained by reversing that of $\mathbb P^n$, so that points
and hyperplanes exchange roles. We denote by $[\ell]$ the point
of $(\mathbb P^n)^\vee$ corresponding to the hyperplane
$Z(\ell)\subseteq\mathbb P^n$, where $\ell$ is a nonzero linear
form, determined up to multiplication by a nonzero scalar.

\subsection{Fiber dimension and generic fibers} 
\ \medskip

Recall that for a projective variety $X$ a function $f:X \rightarrow \mathbb{A}^1$ is regular if it can be written locally in affine charts as a quotient of polynomials with a nonvanishing denominator. A \emph{morphism of varieties} $\varphi:X\to Y$ is a continuous map in the Zariski topology such that on every open set $U \subseteq Y$ the composition of every regular function on $U$ with $\varphi$ is regular in $\varphi^{-1}(U)$. We now state some useful results about fibers of morphisms that we will need later.

\begin{theorem}
\label{teorema de la dimensión de la fibra}
Let $V\subseteq\mathbb A^n$ and $W\subseteq\mathbb A^m$ be irreducible
varieties, let $G\subseteq V$ be a dense open subset, and let
$\varphi\colon G\to W$ be a dominant morphism. Then
\[
    \dim \varphi^{-1}(p)\geq \dim V-\dim W
\]
for every $p\in\varphi(G)$. Moreover, there exists a nonempty open subset
$U\subseteq W$ such that equality holds for every $p\in U\cap\varphi(G)$.
\end{theorem}

\begin{proof}
See \cite[Chapter~I, Section~6.3, Theorem~1.25]{shafarevich2013basic}.
\end{proof}

\begin{corollary}\label{coro:fibra}
Let $f\colon V\to W$ be a dominant morphism between irreducible affine
varieties of the same dimension. There exists a nonempty open subset
$U\subseteq W$ such that, for every $y\in U$, the fiber $f^{-1}(y)$ is
finite and
\[
    \#f^{-1}(y)=[K(V):K(W)]_f.
\]
\end{corollary}

\begin{proof}
See \cite[Corollary~6]{jeronimo2025geometric}.
\end{proof}

\subsection{Tangent spaces and associated varieties} \ \medskip

In this subsection, we introduce the different notions of tangent
space and fix the notation used throughout the paper. We then define the tangential, secant, and dual varieties.
For background on smoothness, we refer the reader to
\cite{matsumura1980commutative,eisenbud1995commutative}.

\begin{notation}
For an affine variety $V\subseteq\mathbb A^n$ and a projective
variety $X\subseteq\mathbb P^n$, we denote their smooth loci by
$V_{\operatorname{reg}}$ and $X_{\operatorname{reg}}$, respectively.
\end{notation}

The Jacobian criterion (see \cite[\S.29]{matsumura1980commutative}) easily shows that $X_{\operatorname{reg}}$ and $V_{\operatorname{reg}}$ are Zariski open dense subsets of $X$ and $V$ respectively.

\begin{definition}
Let $V\subseteq\mathbb A^n$ be an affine variety, let
$X\subseteq\mathbb P^n$ be its projective closure, and let $p\in V_{\operatorname{reg}}$.

\begin{itemize}
    \item The {Zariski tangent space} to $V$ at $p$ (equivalently, to $X$ at $[1:p]$) is
    \[
        T_pV
        :=
        \left\{
            v\in K^n
            \;\middle|\;
            \sum_{i=1}^n
            \frac{\partial f}{\partial x_i}(p)v_i=0
            \text{ for every }f\in I(V)
        \right\} =: T_{[1:p]}X.
    \]

    \item The {affine tangent space} to $V$ at $p$ is
    \[
        T_pV+p
        :=
        \{v+p\mid v\in T_pV\}
        \subseteq\mathbb A^n.
    \]

    \item The {projective tangent space} to $V$ at $p$ (equivalently, to $X$ at $[1:p]$) is
    \[
        \mathbb T_pV
        :=
        \overline{p+T_pV}
        =
        \mathbb P\left(
            \operatorname{Span}((1,p))
            \oplus
            \bigl(\{0\}\times T_pV\bigr)
        \right):= \mathbb{T}_{[1:p]}X.
    \]
    
\end{itemize}
For any $q\in X_{\operatorname{reg}}$, the spaces $T_qX$ and
$\mathbb T_qX$ are defined analogously using any affine chart
containing $q$.
\end{definition}

Using this notation, we introduce three classical projective
varieties associated with $X$: the tangential, dual, and secant
varieties.

\begin{definition}\label{def tangt}
Let $X\subsetneq\mathbb P^n$ be an irreducible projective variety
of dimension $d \geq 1$.
\begin{itemize}
    \item The {tangential variety} of $X$, denoted by $\Tan(X) \subseteq \mathbb{P}^n$ is: $$\displaystyle \operatorname{Tan}(X)
        :=
        \overline{
            \bigcup_{p\in X_{\operatorname{reg}}}\mathbb T_pX
        }.$$

    \item The {dual variety} of $X$, denoted by $X^\vee \subseteq (\mathbb{P}^n)^\vee$ is
    \[
        X^\vee
        :=
        \overline{
            \left\{
                H\in(\mathbb P^n)^\vee
                \;\middle|\;
                \mathbb T_pX\subseteq H
                \text{ for some }p\in X_{\operatorname{reg}}
            \right\}
        },
    \]
    \item The {secant variety} of $X$, denoted by $\operatorname{Sec}(X) \subseteq \mathbb{P}^n$ is $$\operatorname{Sec}(X)
        := \displaystyle
        \overline{
            \bigcup_{\substack{p,q\in X\\p\neq q}}
            \langle p,q\rangle
        }.$$
      \item The {Gauss map} of $X$ is the morphism
$g:X_{\operatorname{reg}}\longrightarrow\mathbb G(n+1,\dim(X)+1),$
which assigns to each smooth point $p\in X$ the $(d+1)$-dimensional
linear subspace of $K^{n+1}$ whose projectivization is
$\mathbb T_pX$.
\end{itemize}
\end{definition}

We now recall a classical result relating the tangential and
secant varieties of a smooth projective variety.

\begin{theorem}\label{Teorema Zak}
Let $X\subseteq\mathbb P^n$ be a smooth irreducible non degenerate
projective variety of dimension $d$. Then either
\[
    \dim\Tan(X)=2d
    \text{ and }
    \dim\operatorname{Sec}(X)=2d+1,
\text{ or }
    \Tan(X)=\operatorname{Sec}(X).
\]
\end{theorem}
\begin{proof}
See \cite[Chapter~I, Theorem~1.4]{Zak1993}.
\end{proof}

\subsection{Geometric degree}\label{subsec:deg geom}
\ \medskip 

In this subsection, we recall the notion of geometric degree using the approach from \cite{Heintz1983}.

\medskip

Fix $d\leq n$. A pair
$(H,\xi)\in\mathbb A^{d\times n}\times\mathbb A^d$, with
$\operatorname{rank}(H)=d$, determines the affine linear subspace
\[
    H_\xi:=\{p\in\mathbb A^n\mid Hp=\xi\}
\]
of dimension $n-d$. Let $V\subseteq\mathbb A^n$ be irreducible of dimension $d$, and consider
the morphism
\[
F\colon\mathbb A^{d\times n}\times V
\longrightarrow
\mathbb A^{d\times n}\times\mathbb A^d,
\qquad
(G,p)\longmapsto(G,Gp).
\]
For every $(H,\xi)$,
\[
    F^{-1}(H,\xi)=\{H\}\times(V\cap H_\xi).
\]
It can be easily seen that $F$ is
dominant and generically finite (see \cite{Heintz1983}). This gives the following equivalent
descriptions of the geometric degree.

\begin{definition}\label{def: grado geom}
The geometric degree of $V$ is
\[
\deg(V)
:=
\left[
K(\mathbb A^{d\times n}\times V):
K(\mathbb A^{d\times n}\times\mathbb A^d)
\right]_F
=\#(V\cap H_\xi),
\]
where, $H_\xi$ is a \emph{generic} affine linear subspace  of dimension $n-d$.
\end{definition}

Equivalently, $\deg(V)$ is the maximum finite cardinality obtained by
intersecting $V$ with affine linear subspaces of complementary dimension;
see \cite[Lemma~1 and Proposition~1]{Heintz1983}. A linear variety has degree
$1$, the degree of a hypersurface is the degree of its square-free defining
polynomial

For a reducible variety, we define the degree as the sum of the degrees of
its irreducible components. In the nonequidimensional case this convention
does not admit a description by a single general linear section, but it is
the convention used in the affine B\'ezout inequalities below.

\begin{proposition}\label{prop:bezout}
Let $V,V_1,\ldots,V_r\subseteq\mathbb A^n$ be algebraic varieties. Then
\[
\deg(V\cap V_1\cap\cdots\cap V_r)
\leq
\deg(V)\min\left\{
\prod_{i=1}^r\deg(V_i),
\left(\max_{1\leq i\leq r}\deg(V_i)\right)^{\dim V}
\right\}.
\]
In particular, if $f_1,\ldots,f_r\in K[\mathbf x] \setminus \{0\}$, then
\[
\deg Z(f_1,\ldots,f_r)
\leq
\min\left\{
\prod_{i=1}^r\deg(f_i),
\left(\max_{1\leq i\leq r}\deg(f_i)\right)^n
\right\}.
\]
\end{proposition}

\begin{proof}
See \cite[Theorem~1]{Heintz1983} and
\cite[Proposition~2.3]{HS82}.
\end{proof}

The notion of degree introduced above for affine varieties extends naturally to the projective and quasi-projective settings. If \(X\subseteq \mathbb{P}^n\) is a projective variety and $V:= \{x_0 \neq 0\} \cap X $ is a dense affine open subset, we define
\[
\deg(X):=\deg(V),
\]
which geometrically equals the number of intersection points of \(X\) with a general collection of projective hyperplanes of complementary codimension. We finish this section with a deformation Lemma that will be a key in the proof of Theorem \ref{teorema cuadrado} (Theorem C of the introduction). For a quasi-projective variety the degree can be defined as the degree of its closure.

The notion of degree extends naturally to modules and projective schemes via the Hilbert polynomial. We refer the reader to \cite[Chapter I, §7 and Chapter III, Exercise 5.2]{hartshorne1977algebraic} for the relevant definitions and basic properties. 

\begin{lemma}\label{Lema defomracion}
Let \(Z\subseteq\mathbb A^n\times\mathbb A^1\) be an irreducible
quasi-affine algebraic variety of dimension \(d+1\), and let
 $\vartheta\colon Z\longrightarrow\mathbb A^1$
be the projection onto the last coordinate. Assume that \(\vartheta\) is
surjective and that every fiber \(\vartheta^{-1}(s)\) is irreducible of
dimension \(d\). Then
\[
    \deg\bigl(\vartheta^{-1}(0)\bigr)
    \leq
    \deg\bigl(\vartheta^{-1}(s)\bigr) \text{ for general }s\in\mathbb A^1.
\]
\end{lemma}

\begin{proof}
Embed \(\mathbb A^n\) into \(\mathbb P^n\), and let
\[
    \mathcal Z:=\overline Z
    \subseteq\mathbb P^n\times\mathbb A^1.
\]
The projection \(\vartheta\) extends to a projective morphism
\[
    \overline\vartheta\colon
    \mathcal Z\longrightarrow\mathbb A^1.
\]
Since \(\mathcal Z\) is irreducible and
\(\overline\vartheta\) is dominant, the morphism
\(\overline\vartheta\) is flat by
\cite[III, Proposition~9.7]{hartshorne1977algebraic}. Consequently, its
scheme-theoretic fibers have a common Hilbert polynomial and hence a common
degree, say \(D\).

For \(s\in\mathbb A^1\), denote by \(\mathcal Z_s\) the scheme-theoretic
fiber of \(\overline\vartheta\), by \((\mathcal Z_s)_{\mathrm{red}}\) its
underlying reduced variety, and by
\(\overline{\vartheta^{-1}(s)}\subseteq\mathbb P^n\) the projective closure
of the corresponding fiber of \(Z\). The fibers of
\(\overline\vartheta\) are equidimensional of dimension \(d\) (see
\cite[III, Corollary~9.6]{hartshorne1977algebraic}). Since
\(\overline{\vartheta^{-1}(s)}\) is irreducible of dimension \(d\), it is
an irreducible component of \((\mathcal Z_s)_{\mathrm{red}}\). Therefore
\begin{equation}\label{eq:deformation-upper-bound}
    \deg\bigl(\vartheta^{-1}(s)\bigr)
    =
    \deg\bigl(\overline{\vartheta^{-1}(s)}\bigr)
    \leq
    \deg(\mathcal Z_s)
    =
    D
\end{equation}
for every \(s\in\mathbb A^1\).

It remains to prove that equality holds in
\eqref{eq:deformation-upper-bound} for general \(s\). Since \(Z\) is a
dense open subset of \(\mathcal Z\), its boundary
\[
    B:=\mathcal Z\setminus Z
\]
has dimension at most \(d\). The fiber dimension theorem therefore gives
\[
    \dim\bigl(B\cap(\mathbb P^n\times\{s\})\bigr)\leq d-1
\]
for general \(s\). Thus no irreducible component of
\((\mathcal Z_s)_{\mathrm{red}}\) is contained in \(B\). Since
\(\vartheta^{-1}(s)\) is irreducible, it follows that, for general \(s\) we have the equality:
\[
    (\mathcal Z_s)_{\mathrm{red}}
    =
    \overline{\vartheta^{-1}(s)}
\]

Finally, the generic fiber of \(\overline\vartheta\) is integral and hence,
in characteristic zero, geometrically reduced. Openness of geometric
reducedness in a flat morphism of finite presentation
\cite[Théorème~12.2.4(v)]{EGAIV3} implies that \(\mathcal Z_s\) is reduced
for general \(s\). Hence
\[
    D
    =
    \deg(\mathcal Z_s)
    =
    \deg\bigl(\overline{\vartheta^{-1}(s)}\bigr)
    =
    \deg\bigl(\vartheta^{-1}(s)\bigr)
\]
for general \(s\). Applying
\eqref{eq:deformation-upper-bound} to \(s=0\) proves the result.
\end{proof}

\subsection{The biprojective interpretation and its bidegrees}
\ \medskip

In this section we introduce the bidegrees associated to a biprojective variety, following the strategy from \cite{VANDERWAERDEN1978303}.

\begin{definition}\label{def X'}
Let $X\subseteq\mathbb P^{n+m}$ be an irreducible projective variety whose
ideal $I(X)\subseteq K[\overline{\mathbf x},\mathbf y]$ is bihomogeneous.
The \emph{biprojective interpretation} of $X$ is the variety
\[
X':=
\left\{
([\overline{\mathbf x}],[\mathbf y])
\in\mathbb P^n\times\mathbb P^{m-1}
\ \middle|\
F(\overline{\mathbf x},\mathbf y)=0
\text{ for every }F\in I(X)
\right\}.
\]
\end{definition}

Provided that $X$ is not contained in either coordinate subspace, $X'$ is
equidimensional and
\[
    \dim X'=\dim X-1.
\]

The following lemma explains why this construction will apply to the varieties considered in Section \ref{section emb}.

\begin{lemma}\label{asd}
Let $Y\subseteq\mathbb A^{n+m}$ be an algebraic variety such that
$(\mathbf x,\lambda\mathbf y)\in Y$ for every
$(\mathbf x,\mathbf y)\in Y$ and every $\lambda\in K$. If
$\overline Y\subseteq\mathbb P^{n+m}$ is the projective closure of $Y$, then
$I(\overline Y)\subseteq K[\overline{\mathbf x},\mathbf y]$ is
bihomogeneous.
\end{lemma}

\begin{proof}
See \cite[Lemma 22]{jeronimo2025geometric}.
\end{proof}

This bihomogeneous structure allows us to pass from the affine to the corresponding biprojective setting where we will use bidegrees to study the geometric degree of embedded algebraic vector bundles.

\begin{definition}\label{def:bidegree}
Let $X\subseteq\mathbb P^{n+m}$ be an irreducible projective variety whose ideal $I(X)$ is bihomogeneous and
$\dim X=d$. Let $\alpha,\beta\in\mathbb N_0$ satisfy $\alpha+\beta=d-1$. The
\emph{bidegree} $\deg_{\alpha,\beta}(X')$ is
\[
    \deg_{\alpha,\beta}(X')
    :=\#\bigl(X'\cap(L_1\times L_2)\bigr),
\]
where $L_1\subseteq\mathbb P^n$ and
$L_2\subseteq\mathbb P^{m-1}$ are general linear subspaces of dimensions
$n-\alpha$ and $m-1-\beta$, respectively. 
\end{definition}
We set $\deg_{\alpha,\beta}(X')=0$ whenever
$\alpha>n$ or $\beta>m-1$. Van der Waerden's formula relates the bidegree of $X'$ to the geometric degree of the projective variety $X$.

\begin{theorem}[Van der Waerden]\label{thm:vdw}
Let $X\subseteq\mathbb P^{n+m}$ be an irreducible projective variety defined
by a bihomogeneous ideal such that $X$ is not contained in either coordinate subspace, and let
$X'\subseteq\mathbb P^n\times\mathbb P^{m-1}$ be its biprojective
interpretation. Then
\begin{equation*}\label{eq:multideg}
    \deg(X)=
    \sum_{\alpha+\beta=\dim X-1}\deg_{\alpha,\beta}(X').
\end{equation*}
\end{theorem}

\begin{proof}
See \cite{VANDERWAERDEN1978303}.
\end{proof}

\section{Embedded Algebraic Vector Bundles and Direction Varieties}\label{section emb}

This section introduces embedded algebraic vector bundles and their
associated direction varieties. We establish the basic properties that will
be used throughout the paper.

\subsection{Embedded Algebraic Vector Bundles}
 \
 \medskip

The vector bundles considered in this paper will come equipped with a fixed embedding into a trivial ambient bundle. This additional structure will allow us to regard each fiber as a linear subspace of a fixed vector space and, in particular, to study how these subspaces vary with the base point. We begin by formalizing this notion.

\begin{definition}\label{definicion 1}
Let $V\subseteq\mathbb A^n$ be an irreducible algebraic variety and let
$E\subseteq\mathbb A^{n+m}$ be an algebraic variety. Let
$\pi_1\colon E\to V$ be the restriction of the projection onto the first
$n$ coordinates. We say that $E$ is an \emph{embedded algebraic vector bundle over
$V$} if the following conditions hold:
\begin{enumerate}
    \item\label{cond:fibrado-fibras} For every $p\in V$,
    \[
        \pi_1^{-1}(p)=\{p\}\times L_p,
    \]
    where $L_p\subseteq\mathbb A^m$ is a linear subspace of constant
    dimension $r$.
    \item\label{cond:fibrado-trivialidad-local} The morphism
    $\pi_1\colon E\to V$ is Zariski locally trivial: every point of $V$ has
    an open neighborhood $U$ for which there is an isomorphism
    \[
        \pi_1^{-1}(U)\simeq U\times\mathbb A^r
    \]
    over $U$ whose restriction to each fiber is linear.
\end{enumerate}
Under these conditions, we say that $E$ has rank $r$ and write
$\operatorname{rk}(E)=r$.
\end{definition}

This is the affine specialization of the notion of a geometric vector bundle
introduced in \cite[II, Exercise~5.18]{hartshorne1977algebraic}.  Unless otherwise stated, all embedded algebraic vector bundles considered
below have positive rank.

\begin{example}\label{ex:trivially-embedded-bundle}
Let $V\subseteq\mathbb A^n$ be an irreducible algebraic variety and let
$L\subseteq\mathbb A^m$ be a linear subspace of dimension $r$. The variety
$V\times L$ is an embedded algebraic vector bundle of rank $r$ over $V$. We call it a
\emph{trivially embedded bundle}.
\end{example}

The following observation gives a practical criterion to determine if a variety is an embedded algebraic vector bundle or not.

\begin{observation}\label{obs:localmente-trivial}
Let $V\subseteq\mathbb A^n$ be an irreducible algebraic variety, and let
\[
E=
Z\bigl(
I(V)+
\langle
\ell_1(\mathbf x,\mathbf y),\ldots,
\ell_s(\mathbf x,\mathbf y)
\rangle
\bigr)
\subseteq \mathbb A^{n+m},
\]
where the $\ell_i\in K[\mathbf x,\mathbf y]$ are homogeneous linear forms
in the variables $\mathbf y$. If the fibers of the projection $\pi_1:E\longrightarrow V$
have constant dimension $r$, then $E$ is an embedded algebraic vector bundle
of rank $r$ over $V$.
\end{observation}

\begin{proof}
It is clear condition \ref{cond:fibrado-fibras} is satisfied. So we only need to verify condition \ref{cond:fibrado-trivialidad-local} holds. For this, write the equations $\ell_i(\mathbf x,\mathbf y)$ as
\[
    A(\mathbf x)\mathbf y^t=0,
\]
where $A(\mathbf x)$ is an $s\times m$ matrix with entries in
$K[\mathbf x]$. Since every fiber has dimension $r$, one has
\[
    \operatorname{rank}A(p)=m-r
    \qquad\text{for every }p\in V.
\]

Fix $p\in V$ and choose an $(m-r)\times(m-r)$ minor
$h(\mathbf x)$ of $A(\mathbf x)$ such that $h(p)\neq0$. After reordering
the equations and the variables $y_i$, we may assume that this is the minor
corresponding to the last $m-r$ variables. Set
\[
    U:=\{q\in V\mid h(q)\neq0\}.
\]
For every $q\in U$, Cramer's rule gives regular functions
$a_{ij}:U \rightarrow \mathbb{A}^1$ such that the system
$A(q)\mathbf y^t=0$ is equivalent to
\[
    y_{r+j}=\sum_{i=1}^r a_{ij}(q)y_i,
    \qquad 1\leq j\leq m-r.
\]
Consequently, projection onto the first $r$ coordinates of $\mathbf y$
induces an isomorphism
\[
    \pi_1^{-1}(U)\simeq U\times\mathbb A^r
\]
over $U$ which is linear on the fibers. Since $p$ was arbitrary, $E$ is Zariski locally trivial.
\end{proof}
We are now ready to give an important example:
\begin{example}
Let $V\subseteq\mathbb A^n$ be a smooth irreducible algebraic variety. By the Jacobian criterion (see \cite[\S.29]{matsumura1980commutative}) and
Observation~\ref{obs:localmente-trivial}, its tangent bundle
$$TV\subseteq\mathbb A^{2n}, \quad TV :=\{(p,q) \in \mathbb{A}^{2n} \mid p \in V, \ q \in T_p V\}$$ is an embedded algebraic vector bundle of rank
$\dim V$ over $V$.
\end{example}

The dimension of the tangent bundle $TV$ has been studied extensively (see
\cite[Proposition~8]{jeronimo2025geometric} or \cite{Kunz1999} for a more general result). The following proposition will extend \cite[Proposition~8]{jeronimo2025geometric} from tangent bundles to embedded algebraic vector bundles.

\begin{proposition}\label{proposición completa}
Let $V\subseteq\mathbb A^n$ be an irreducible algebraic variety and let $E$
be an embedded algebraic vector bundle over $V$. Then $E$ is irreducible and
\[
    \dim E=\dim V+\operatorname{rk}(E).
\]
\end{proposition}

\begin{proof}
Set $r:=\operatorname{rk}(E)$, and let $C_1,\ldots,C_k$ be the
irreducible components of $E$. Fix $(p,q)\in C_l$. By local triviality,
there is an irreducible open neighborhood $U\subseteq V$ of $p$ such that
\[
    \pi_1^{-1}(U)\simeq U\times\mathbb A^r.
\]
In particular, $\pi_1^{-1}(U)$ is an irreducible open subset of $E$ of
dimension $\dim V+r$, and hence it is contained in an irreducible component
of $E$. Since $\pi_1^{-1}(U)\cap C_l$ is a nonempty open subset of $C_l$,
this component must be $C_l$. Therefore
\[
    \pi_1^{-1}(U)\subseteq C_l
    \quad \text{and} \quad
    \dim C_l=\dim V+r.
\]

The same argument shows that the components are pairwise disjoint. Indeed,
if a point belonged to $C_i\cap C_j$, a trivializing open subset containing
that point would be contained in both components. Since it is nonempty and
open in each of them, this would imply $C_i=C_j$.

For every $p\in V$, the irreducible fiber
\[
    \pi_1^{-1}(p)= \{ p\} \times  L_p
\]
is the finite union of the closed subsets $(\{ p\} \times L_p)\cap C_j$. Hence, $\{p \} \times L_p$ is
contained in one of the components $C_j$. Since the components are pairwise
disjoint, this component is unique. It follows that the constructible sets
$\pi_1(C_j)$ are pairwise disjoint.

On the other hand, every fiber of
$\pi_1|_{C_j}\colon C_j\to V$ has dimension at most $r$. Since
$\dim C_j=\dim V+r$, the fiber dimension theorem gives
\[
    \dim\overline{\pi_1(C_j)}=\dim V.
\]
Thus every $\pi_1|_{C_j}$ is dominant, and Chevalley's theorem (see \cite[Chapter I, \S 5.3, Theorem 1.14]{shafarevich2013basic}) implies that
$\pi_1(C_j)$ contains a dense open subset of $V$. If $k>1$, two such dense
open subsets would intersect, contradicting the pairwise disjointness of the
sets $\pi_1(C_j)$. Therefore $k=1$, and $E$ is irreducible of dimension
$\dim V+r$.
\end{proof}

\subsection{Direction Varieties}
\ \medskip

As noted above, an embedding of the total space of a vector bundle into a
trivial ambient bundle allows us to view its fibers as linear subspaces of a
fixed vector space. This leads naturally to the following fundamental notion.

\begin{definition}\label{def: direcciones}
Let $V\subseteq\mathbb A^n$ be an irreducible algebraic variety and let
$E\subseteq\mathbb A^{n+m}$ be an embedded algebraic vector bundle over $V$. Let
$\pi_2\colon E\to\mathbb A^m$ be the projection onto the
$\mathbf y$-coordinates. We define
\[
    \mathscr{D}(E):=\overline{\pi_2(E)}
\]
and call it the \emph{variety of directions} of $E$.
\end{definition}

\begin{notation}
    For a smooth algebraic variety $V \subseteq \mathbb{A}^n$ we write
    $$\operatorname{Tan}_0(V) = \mathscr{D}(TV),$$
    and call it the variety of tangent directions of $V$.
\end{notation}

The variety of tangent directions was studied in
\cite{jeronimo2025geometric} in the case of smooth curves. It should not be
confused with the tangential variety, which will be defined later.

\medskip

The next lemma shows that the direction variety inherits irreducibility from
the total space and gives sharp bounds for its dimension.

\begin{lemma}\label{lema 7}
Let $V\subseteq\mathbb A^n$ be an irreducible algebraic variety and let
$E\subseteq\mathbb A^{n+m}$ be an embedded algebraic vector bundle of rank $r$ over
$V$. Then $\mathscr{D}(E)$ is an irreducible algebraic variety and
\[
    r\leq\dim\mathscr{D}(E)\leq\min\{r+\dim V,m\}.
\]
Moreover, $\dim\mathscr{D}(E)=r$ if and only if $E$ is trivially embedded.
\end{lemma}

\begin{proof}
By Proposition~\ref{proposición completa}, $E$ is irreducible, and hence so
is $\mathscr{D}(E)=\overline{\pi_2(E)}$. The upper bound follows from
\[
    \dim\mathscr{D}(E)\leq\min\{\dim E,m\}
    =\min\{r+\dim V,m\},
\]
whereas the lower bound follows because $\mathscr{D}(E)$ contains every fiber
$L_p$, each of which has dimension $r$.

If $E$ is trivially embedded, then $E=V\times L$ for some
$r$-dimensional linear subspace $L\subseteq\mathbb A^m$, and therefore
$\mathscr{D}(E)=L$. Conversely, suppose that $\dim\mathscr{D}(E)=r$ and that $E$ is not
trivially embedded. Then there exist $p_1,p_2\in V$ such that, writing
\[
    \pi_1^{-1}(p_i)=\{p_i\}\times L_i,
    \qquad i=1,2,
\]
one has $L_1\neq L_2$. Since each $L_i$ is a closed irreducible subvariety
of $\mathscr{D}(E)$ and
\[
    \dim (L_i)=\dim(\mathscr{D}(E))=r,
\]
each $L_i$ is an irreducible component of $\mathscr{D}(E)$. This contradicts the
irreducibility of $\mathscr{D}(E)$.
\end{proof}

\subsection{Generic Linear Sections}
\ \medskip

In Sections \ref{asdasdasd} and \ref{asdff}, we will repeatedly restrict an
embedded algebraic vector bundle to the intersection of its base with a
general affine linear subspace. In order for the resulting invariants to be
well defined, we need to ensure that the dimension and degree of the
corresponding direction variety are independent of the chosen general
section. In this subsection, we establish this generic constancy first in the biprojective setting (see Lemma \ref{lema tecnico}) and then apply it to embedded algebraic vector bundles.

\begin{notation}
Let $T'\subseteq\mathbb P^n\times\mathbb P^{m-1}$ be a biprojective
variety, and let $\alpha,\beta\in\mathbb N_0$ satisfy $\alpha\leq n$ and $\beta\leq m-1$.
We use the following notation:
\begin{itemize}
    \item We denote the two projections from $\mathbb G(n+1,n+1-\alpha)\times\mathbb G(m,m-\beta)$
    by $\Pi_1$ and $\Pi_2$.

    \item The morphisms
    \[
        \pi_1\colon T'\to\mathbb P^n,
        \qquad
        \pi_2\colon T'\to\mathbb P^{m-1}
    \]
    denote the two natural projections.

    \item For $L\in\mathbb G(n+1,n+1-\alpha)$ and
    $H\in\mathbb G(m,m-\beta)$, we write
    \[
    T'\barwedge L
    :=T'\cap\bigl(\mathbb P(L)\times\mathbb P^{m-1}\bigr),
    \qquad
    T'\barwedge H
    :=T'\cap\bigl(\mathbb P^n\times\mathbb P(H)\bigr).
    \]
\end{itemize}
\end{notation}

A projective $k$-plane in $\mathbb P^n$ can be represented as the common
zero locus of $n-k$ homogeneous linear forms. Equivalently, it is
parametrized by a full-rank matrix in
$\mathbb A^{(n-k)\times(n+1)}$. The natural morphism from the open set of
full-rank matrices to $\mathbb G(n+1,k+1)$ is surjective and open. Hence the
notion of a general linear subspace obtained from matrices agrees with the
intrinsic notion obtained from the Grassmannian.

\begin{lemma}\label{lema tecnico}
Let $T\subseteq\mathbb P^{n+m}$ be an irreducible projective variety defined
by a bihomogeneous ideal, and let
$T'\subseteq\mathbb P^n\times\mathbb P^{m-1}$ be its biprojective
interpretation. Let $\alpha,\beta\in\mathbb N_0$ satisfy
\[
    \alpha+\beta=\dim T-1=\dim T'.
\]
Assume that:
\begin{enumerate}
    \item[(a)] for a general
    $L\in\mathbb G(n+1,n+1-\alpha)$,
    \[
        \dim(T'\barwedge L)=\dim T'-\alpha=\beta;
    \]

    \item[(b)] $\deg_{\alpha,\beta}(T')>0$.
\end{enumerate}
Then there exists a nonempty open subset
$W\subseteq\mathbb G(n+1,n+1-\alpha)$ such that, for every $L\in W$,
 $\dim\pi_2(T'\barwedge L)=\beta,$
and the degree of $\pi_2(T'\barwedge L)$ is independent of $L\in W$.
\end{lemma}

\begin{proof}
By the definition of the bidegree, there exists a nonempty open subset
\[
W_1\subseteq
\mathbb G(n+1,n+1-\alpha)\times\mathbb G(m,m-\beta)
\]
such that
\[
    \deg_{\alpha,\beta}(T')
    =\#(T'\barwedge L\barwedge H)
\]
for every $(L,H)\in W_1$. Since the projection $\Pi_1$ is open,
$\Pi_1(W_1)$ is a nonempty open subset of
$\mathbb G(n+1,n+1-\alpha)$. By hypothesis~(a), there is also a nonempty open
subset $\widehat U\subseteq\mathbb G(n+1,n+1-\alpha)$ such that
\[
    \dim(T'\barwedge L)=\beta
\]
for every $L\in\widehat U$. Set $U_2:=\Pi_1(W_1)\cap\widehat U $
and fix $L\in U_2$. Since the fiber of $W_1$ over $L$ is a nonempty open
subset of $\mathbb G(m,m-\beta)$, one has
\[
    \deg_{0,\beta}(T'\barwedge L)
    =\deg_{\alpha,\beta}(T')>0.
\]
Applying \cite[Theorem~A]{CASTILLO2020107382} to
$T'\barwedge L$, we obtain
\[
    \beta \leq\dim\pi_2(T'\barwedge L).
\]
The reverse inequality follows from
$\dim(T'\barwedge L)=\beta$. Therefore
\[
    \dim\pi_2(T'\barwedge L)=\beta
    \qquad\text{for every }L\in U_2.
\]

We now prove the constancy of the degree. Consider the incidence variety
\[
\mathcal I:=
\left\{
(\xi,(L,H))\in
\pi_2(T')\times\Pi_1^{-1}(U_2)
\ \middle|\
\text{there exists }p\in\mathbb P^n
\text{ such that }(p,\xi)\in T'\barwedge L\barwedge H
\right\}.
\]
The variety $\mathcal I$ is closed because it is obtained by projecting a
closed incidence through the projective $p$-coordinate. Consider the morphism
\[
\Psi\colon\mathcal I\longrightarrow\Pi_1^{-1}(U_2),
\qquad
\Psi(\xi,(L,H))=(L,H).
\]
For every $(L,H)\in\Pi_1^{-1}(U_2)$, its fiber is
\[
    \Psi^{-1}(L,H)
    =\pi_2(T'\barwedge L)\cap\mathbb P(H).
\]

Since $\pi_2(T'\barwedge L)$ has dimension $\beta$ for every $L\in U_2$, the
fibers of $\Psi$ are finite and nonempty for a general pair
$(L,H)\in\Pi_1^{-1}(U_2)$.
Thus $\Psi$ is dominant and generically finite. Discarding the components
that do not dominate $\Pi_1^{-1}(U_2)$ and arguing componentwise, we may
assume without loss of generality that $\mathcal I$ is irreducible. In the
general case, after avoiding the images of the pairwise intersections, the
corresponding generic cardinalities are simply added.

Applying Corollary~\ref{coro:fibra} on suitable affine open subsets and
shrinking the parameter space, there exist a nonempty open subset
$W_3\subseteq\Pi_1^{-1}(U_2)$ and an integer $N$ such that
\[
    \#\Psi^{-1}(L,H)=N
\]
for every $(L,H)\in W_3$. Set
\[
    W:=\Pi_1(W_3).
\]
Since $\Pi_1^{-1}(U_2)$ is a product and $W_3$ is open, $W$ is a nonempty
open subset of $U_2$. For every $L\in W$, the fiber of $W_3$ over $L$ is a
nonempty open subset of $\mathbb G(m,m-\beta)$. Hence, for a general $H$ in
this fiber,
\[
    \#\bigl(\pi_2(T'\barwedge L)\cap\mathbb P(H)\bigr)=N.
\]
By the definition of the degree of a projective variety, it follows that
\[
    \deg\bigl(\pi_2(T'\barwedge L)\bigr)=N
\]
for every $L\in W$.
\end{proof}

We now specialize the preceding lemma to embedded algebraic vector bundles.
Let $V\subseteq\mathbb A^n$ be an irreducible algebraic variety and let
$E\subseteq\mathbb A^{n+m}$ be an embedded algebraic vector bundle over $V$.

Denote by $\overline E\subseteq\mathbb P^{n+m}$ its projective closure and by
$\overline E'\subseteq\mathbb P^n\times\mathbb P^{m-1}$ its biprojective
interpretation. Let $L\subseteq\mathbb A^n$ be a general affine linear
subspace of codimension $\alpha$ with $0 \leq \alpha \leq d$, and denote its projective closure by
$\overline L\subseteq\mathbb P^n$. Then
\[
    \left(\overline{E\cap(L\times\mathbb A^m)}\right)'
    =\overline E'\cap
    \bigl(\overline L\times\mathbb P^{m-1}\bigr),
\]
and this variety is equidimensional of dimension
$\dim E-\alpha-1=\dim\overline E'-\alpha.$
Moreover,
\begin{equation}\label{ecuacion cono}
    \mathscr{D}\bigl(E\cap(L\times\mathbb A^m)\bigr)
    =C\left(
        \pi_2\left(
            \overline E'\cap
            \bigl(\overline L\times\mathbb P^{m-1}\bigr)
        \right)
    \right),
\end{equation}
where $C$ denotes the affine cone over a projective variety. Lemma~\ref{lema tecnico} therefore gives the following consequence:
\medskip

\begin{corollary}\label{observación del lema tecnico}
Let $V\subseteq\mathbb A^n$ be an irreducible algebraic variety and let
$E\subseteq\mathbb A^{n+m}$ be an embedded algebraic vector bundle over $V$. Let
$\alpha,\beta\in\mathbb N_0$ satisfy $\alpha+\beta=\dim V+\operatorname{rk}(E)-1$
and assume that $\deg_{\alpha,\beta}(\overline E')>0$. Then the dimension and degree
of
\[
    \mathscr{D}\bigl(E\cap(L\times\mathbb A^m)\bigr)
\]
are independent of the general affine linear subspace
$L\subseteq\mathbb A^n$ of dimension $n-\alpha$.
\end{corollary}

\begin{proof}
Apply Lemma~\ref{lema tecnico} to $T=\overline E$ and
$T'=\overline E'$, and use
\eqref{ecuacion cono}, together with the fact that taking the affine cone
does not change the degree.
\end{proof}

This motivates the following notation.

\begin{notation}\label{notacion Ea}
Let $E\subseteq\mathbb A^{n+m}$ be an embedded algebraic vector bundle of rank $r$ over an irreducible variety $V\subseteq\mathbb A^n$ of dimension $d$, and
let $1\leq a\leq d$. We write
\[
    E_a:=E\cap(L\times\mathbb A^m)
\]
for a general affine linear subspace
$L\subseteq\mathbb A^n$ of dimension $n-(d-a)$, and set
\[
    V_a:=V\cap L.
\]
\end{notation}

The restriction $E_a$ is an embedded algebraic vector bundle of rank $r$ over
$V_a$. Moreover, Bertini's theorem implies that $V_a$ is irreducible and
$\dim (V_a)=a$ (see \cite{jouanolou1983,flenner1999joins}). Consequently, by Proposition~\ref{proposición completa} we have: $\dim( E_a)=a+r.$

\section{A General Defect Theory}\label{asdff}

Let \(V\subseteq\mathbb A^n\) be an irreducible algebraic variety of
dimension \(d\), and let
\(E\subseteq\mathbb A^{n+m}\) be an embedded algebraic vector bundle of rank
\(r\) over \(V\). Its variety of directions
\[
    \mathscr{D}(E)=\overline{\pi_2(E)}
\]
is given by the Zariski closure of the union of the linear spaces \(L_p\subseteq\mathbb A^m\), as \(p\)
varies in \(V\). Lemma~\ref{lema 7} gives
\[
    r\leq\dim\mathscr{D}(E)\leq\min\{d+r,m\}.
\]
The upper bound \(\min\{d+r,m\}\) is the natural expected dimension of
\(\mathscr{D}(E)\): it is the maximal dimension that a \(d\)-dimensional family of
\(r\)-dimensional linear spaces can sweep inside \(\mathbb A^m\). 

\begin{definition}\label{defect definition}
Let $V \subseteq \mathbb{A}^n$ be an irreducible algebraic variety of dimension $d$ and $E \subseteq \mathbb{A}^{n+m}$ be an embedded algebraic vector bundle over $V$ of rank $r$.
\begin{itemize}
        \item The directional defect of \(E\) is
\[
    \Delta(E):=\min\{d+r,m\}-\dim\mathscr{D}(E).
\]
    \item The Grassmann classifying map of \(E\) is the regular map 
\[
\Gamma_E\colon
V\longrightarrow\mathbb G(m,r),
\qquad
p\longmapsto L_p,
\]

\item The Grassmann defect of \(E\) is
\[
    \Delta_G(E)
    :=
    d-\dim\overline{\Gamma_E(V)}.
\]
\end{itemize}

We say that \(E\) is \emph{defective} if \(\Delta(E)>0\) and  {Grassmann defective} if \(\Delta_G(E)>0\).
\end{definition}
Thus, \(\Delta(E)\) measures the failure of the fibers to sweep out a variety in $\mathbb{A}^m$ of the expected dimension while $\Delta_G(E)$ measures the failure of the fibers of $E$ regarding them as points inside the Grassmannian.

\medskip

The section is organized as follows. In Subsection~\ref{subsec:swept-linear},
we show that the defect theory of embedded algebraic vector bundles is
equivalent to the defect theory of varieties swept out by linear spaces
(see Definition~\ref{def: scoper} and Proposition~\ref{defectia}).
In Subsection~\ref{subsec: effective deff}, we establish the numerical
criterion for directional defectivity, Theorem~\ref{thm defect}, stated as
Theorem A in the introduction. In
Subsection~\ref{subsec:Gauss-bundle-duality}, we compare directional and
Grassmann defectivity and study the relation between the two notions.

\subsection{Varieties swept out by linear spaces}\label{subsec:swept-linear}

\begin{definition}\label{def: scoper}
    Let $X \subseteq \mathbb{P}^m$ be an irreducible projective variety.
    \begin{itemize}
        \item  We say $X$ is \emph{swept out by linear spaces} if there exist integers $n \in \mathbb{N}$, $k \in \mathbb{N}_0$, an irreducible quasi-projective variety $B \subseteq \mathbb{P}^n$ and a regular map $\gamma_X: B \rightarrow \mathbb{G}(m+1,k+1)$ such that: $$X = \overline{\left \{ [v] \in \mathbb{P}^m \mid \exists \ p \in B , v \in \gamma_X(p) \setminus \{0\} \right \}} = \overline{\bigcup_{p \in B} \mathbb{P}(\gamma_X(p))},$$ where the bars denote the Zariski closure.
        \item  We define the defect of \((X,\gamma_X)\), by
$$
\delta(X,\gamma_X):=
\min\{\dim B+k,m\}-\dim X.
$$
    \end{itemize}
\end{definition}

Notice that this formulation also captures families of
translated affine linear spaces. Indeed, every affine $k$-plane
$q+L\subseteq\mathbb A^m$ determines the $(k+1)$-dimensional vector
subspace
\[
\langle (1,q) \rangle\oplus\bigl(\{0\}\times L\bigr)
\subseteq\mathbb A^{m+1},
\]
whose projectivization is the projective closure of $q+L$. Conversely,
every $(k+1)$-dimensional vector subspace not contained in
$\{y_0=0\}$ determines, after dehomogenization, a translated affine
$k$-plane.

\begin{observation}\label{rmk: obse}
Let $X \subsetneq \mathbb{P}^n$ be an irreducible projective variety. Then the tangential variety $\Tan(X),$ the secant variety $\operatorname{Sec}(X),$ and the dual variety $X^\vee$ (see Definition \ref{def tangt}) are swept out by linear spaces in the sense of Definition \ref{def: scoper}.
\end{observation}
\begin{proof}
For the tangential variety, take $B=X_{\operatorname{reg}}$
and $\gamma_{\Tan(X)}=g$, the Gauss map.

For the dual variety, take the same base and
$\gamma_{X^\vee}=\iota\circ g$, where $\iota$ is the
isomorphism of Grassmannians sending a linear subspace
$L\subseteq K^{n+1}$ to its annihilator
$L^\perp\subseteq(K^{n+1})^*$.

For the secant variety, take
$B=(X\times X)\setminus\Delta$, where $\Delta$ denotes the
diagonal, and let $\gamma_{\operatorname{Sec}(X)}$ send $(p,q)$ to the
two-dimensional linear subspace whose projectivization is
$\langle p,q\rangle$.
\end{proof}

In particular it is easy to see (following the notation of the previous proof) the defect notions $\delta(\operatorname{Tan}(X),g),$ $\delta(X^\vee,\iota \circ g)$ and $\delta(\operatorname{Sec}(X),\gamma_{\operatorname{Sec}(X)})$ match with the corresponding defect notions from the literature (see \cite{BaurDraismaDeGraaf2007,ChiantiniCiliberto2010, https://doi.org/10.1112/blms.12379}).

The following proposition shows that the defect theory of
varieties swept out by linear spaces can be equivalently
formulated in terms of embedded algebraic vector bundles.

\begin{proposition}\label{defectia}
Let $X\subseteq\mathbb P^m$ be an irreducible projective variety and
let $0\leq k\leq m$. The following conditions are equivalent:
\begin{enumerate}
    \item $X$ is swept out by $k$-dimensional linear spaces.

    \item There exists an irreducible affine variety
    $V\subseteq\mathbb A^n$ and an embedded algebraic vector bundle $E\subseteq V\times\mathbb A^{m+1}$
    of rank $k+1$ such that $\mathscr{D}(E)=C(X)$.
\end{enumerate}

Moreover, under this correspondence: 
$$\delta(X,\gamma_X) = \Delta(E), \qquad \gamma_X|_{V} = \Gamma_E.$$
\end{proposition}
\begin{proof}
Assume first that $X$ is swept out by $k$-dimensional linear spaces,
and let
\[
\Gamma:B\longrightarrow\mathbb G(m+1,k+1)
\]
be its corresponding regular map. Choose a dense affine open subset
$V\subseteq B$ and define
\[
E_\Gamma
:=
\left\{
(p,v)\in V\times\mathbb A^{m+1}
\;\middle|\;
v\in\Gamma(p)
\right\}.
\]
On the inverse image of each standard affine chart of the
Grassmannian, the subspace $\Gamma(p)$ is represented by the columns
of a matrix whose entries are regular functions of $p$. Hence $E_\Gamma$ is an embedded algebraic vector bundle over $V$ by Observation \ref{obs:localmente-trivial}. Since $V$ is dense in $B$,
\[
\mathscr{D}(E_\Gamma)
=
\overline{\bigcup_{p\in V}\Gamma(p)}
=
C\left(
\overline{\bigcup_{p\in B}
\mathbb P\bigl(\Gamma(p)\bigr)}
\right)
=
C(X).
\]

Conversely, let $E\subseteq V\times\mathbb A^{m+1}$
be an embedded algebraic vector bundle of rank $k+1$ such that
$\mathscr{D}(E)=C(X)$, and let
\[
\Gamma_E:V\longrightarrow\mathbb G(m+1,k+1)
\]
be its Grassmann direction map. The implication follows taking $B = V$ and $\gamma_X = \Gamma_E$ and noticing that the cone and the projectivization are inverse operations. The last equalities follow immediately by Definition \ref{defect definition}.
\end{proof}

\begin{example}\label{ejemplo fibrados}
Let $X\subsetneq\mathbb P^n$ be an irreducible projective variety
of dimension $d>0$, and assume that its affine chart
$V:=X\cap\{x_0\neq 0\}$ is nonempty and smooth.
Applying the correspondence of Proposition~\ref{defectia}
to the families from Observation \ref{rmk: obse} yields the
following embedded algebraic vector bundles:
\begin{itemize}
    \item For the tangential variety:
    \[
        \mathbb{T}V_{\mathrm{cone}}
        :=
        \left\{
            \bigl(p,(\lambda,\lambda p+v)\bigr)
            \in V\times\mathbb A^{n+1}
            \;\middle|\;
            \lambda\in K,\ v\in T_pV
        \right\}.
    \]

    \item For the dual variety:
    \[
        N^*(V)
        :=
        \left\{
            (p,\ell)\in V\times(K^{n+1})^*
            \;\middle|\;
            \ell\in g([1:p])^\perp
        \right\}.
    \]

    \item For the secant variety, choose a nonempty affine open
    subset $U\subseteq(V\times V)\setminus\Delta_V$.
    For instance, one may take $U=\{(p,q)\mid p_i\neq q_i\}$
    whenever $f:V\rightarrow \mathbb{A}^1, \ f(\mathbf{x})=x_i$ is nonconstant.
    The secant map then gives the rank-two bundle
    \[
        \mathcal S(V)
        :=
        \left\{
            \bigl((p,q),\lambda(1,p)+\mu(1,q)\bigr)
            \in U\times\mathbb A^{n+1}
            \;\middle|\;
            \lambda,\mu\in K
        \right\}.
    \]
    Notice the smoothness of $V$ is not required to do this.
\end{itemize}
Their varieties of directions are the affine cones over
$\Tan(X)$, $X^\vee$, and $\operatorname{Sec}(X)$, respectively. 
\end{example}

The scope of this framework extends beyond the three constructions considered above. Indeed, several other classical varieties arise naturally as varieties swept out by linear spaces in the sense of Definition~\ref{def: scoper}. Some remarkable examples are: higher secant varieties \(\operatorname{Sec}_s(X)\)  (see, for instance, \cite{ChiantiniCiliberto2010}), osculating varieties  (see \cite{BallicoFontanari2004}) and more generally, joins varieties (see \cite[Chapter~4]{flenner1999joins}).

\subsection{Effective defectivity Criteria}\label{subsec: effective deff}
\ \medskip

In this subsection, we establish an effective numerical criterion characterizing the defectivity of an algebraic vector bundle, proving Theorem~A from the introduction (see Theorem \ref{thm defect} below). We begin with a Lemma.

\begin{lemma}\label{observación positividad}
Let $V\subseteq\mathbb A^n$ be an irreducible algebraic variety of dimension
$d$, and let $E\subseteq\mathbb A^{n+m}$ be an embedded algebraic vector bundle of
positive rank $r$ over $V$. For $1\leq a\leq d$, the following conditions
are equivalent:
\begin{enumerate}
    \item $\deg_{d-a,a+r-1}(\overline E')>0$.

    \item $\dim(\mathscr{D}(E_a))=\dim (E_a)=a+r$.
\end{enumerate}
\end{lemma}

    \begin{proof}
If $a+r>m$, the bidegree vanishes as we are cutting with more hyperplanes than the second factor dimension, while
\[
    \dim\mathscr{D}(E_a)\leq m<a+r,
\]
so the equivalence holds. We may therefore assume that $a+r\leq m$.
    
By the definition of $E_a$ and the genericity of the linear section,
\[
    \deg_{d-a,a+r-1}(\overline E')
    =
    \deg_{0,a+r-1}(\overline E_a').
\]
Moreover, $\overline E_a'$ is equidimensional of dimension $a+r-1$ and,
by~\eqref{ecuacion cono}, we have $\mathscr{D}(E_a)=C\bigl(\pi_2(\overline E_a')\bigr).$
Therefore
\[
    \dim\mathscr{D}(E_a)=a+r
    \quad\Longleftrightarrow\quad
    \dim\pi_2(\overline E_a')=a+r-1.
\]
The equivalence now follows from
\cite[Theorem~A]{CASTILLO2020107382}.
\end{proof}

In particular if we consider $a = d$ in Lemma \ref{observación positividad} then Theorem A follows from embedded algebraic vector bundles $E$ in the stable range $d+r \leq m$. The following proposition will let us extend this result to the whole range:

\begin{proposition}
\label{prop:recovery-successive-sections}
Let \(V\subseteq\mathbb A^n\) be an irreducible algebraic variety of
dimension \(d\), and let
\(E\subseteq\mathbb A^{n+m}\) be an embedded algebraic vector bundle of
rank \(r\) over \(V\). Assume that $\Delta(E)<\min\{d,m-r\}$. Then
\[
    \mathscr{D}(E_a)=\mathscr{D}(E)
    \quad\text{for every } a \in \mathbb{N}_0 \text{ such that }
    \min\{d,m-r\}-\Delta(E)\leq a\leq d.
\]
In particular, the corresponding bidegrees from
Lemma~\ref{observación positividad} vanish for
$a>\min\{d,m-r\}-\Delta(E)$, whereas the one corresponding
to $a=\min\{d,m-r\}-\Delta(E)$ is nonzero.
\end{proposition}

\begin{proof}
For a generic \(q\in\mathscr{D}(E)\), set
\[
    F_q:=\{p\in V\mid q\in L_p\}.
\]
By Theorem~\ref{teorema de la dimensión de la fibra},
as $q$ is generic we have:
\[
    \dim F_q
    =d+r-\dim\mathscr{D}(E)
    =d-\min\{d,m-r\}+\Delta(E).
\]
Fix an integer $0\leq c\leq\dim F_q$. Consider $U \subseteq \mathbb{A}^{nc} \times \mathbb{A}^c$ an open dense subset such that $U$ is a parameter space of affine linear subspaces of codimension $c$ (see Subsection \ref{subsec:deg geom}). Now consider the incidence
\[
  \mathcal{I}=  \left\{
    ((p,q),L)\in E\times U
    \ \middle|\
    p\in L
    \right\}.
\]
Let $\Psi:\mathcal{I} \rightarrow \mathscr{D}(E)\times U $ be the projection map and $(q,L) \in \mathscr{D}(E) \times U$ then: $$\Psi^{-1}(q,L) = (F_q \cap L) \times \{q \} \times \{L\}$$

 Notice for generic $q \in \mathscr{D}(E)$ we have that for generic $L \in U$ the variety \(F_q \cap L\) is nonempty and has dimension \(\dim F_q-c\). Therefore this projection is dominant and hence, by Chevalley's Theorem (see \cite[Chapter I, \S 5.3, Theorem 1.14]{shafarevich2013basic}) we can find a dense open set $G \subseteq \operatorname{Im}(\Psi) \subseteq \mathscr{D}(E) \times U$. If we now project $G$ onto $U$ and call this set $\Omega,$ it is easy to see that as the projection is open $\Omega$ is open and if $L  \in \Omega$ then  the following condition is satisfied:
 \begin{equation}\label{ecuacion G_L}
     G_L:= \{ q \in \mathscr{D}(E) \mid  (q,L) \in G\} \subseteq \pi_2(E \cap (L \times \mathbb{A}^m)) \subseteq \mathscr{D}(E).
 \end{equation}

Notice if we consider the mapping $i_L: \mathscr{D}(E) \rightarrow \mathscr{D}(E) \times U,$ given by: $i_L(q) = (q,L)$ we have that $G_L= i_L^{-1}(G)$ and hence is open and dense in $\mathscr{D}(E)$.

Taking the Zariski closure on \eqref{ecuacion G_L} it follows:
$$\mathscr{D}(E) \subseteq \mathscr{D}(E_{d-c}) \subseteq \mathscr{D}(E) \text{ for every } 0 \leq c \leq \dim(F_q) = d-\min\{d,m-r\}+\Delta(E).$$
This proves the asserted equalities. The assertion about the vanishing of the bidegrees follows directly from
Lemma~\ref{observación positividad}.
\end{proof}

Lemma~\ref{observación positividad} gives an effective numerical
criterion for detecting defectivity in the stable range.

\begin{theorem}[Theorem A from the Introduction]\label{thm defect}
Let $V\subseteq\mathbb A^n$ be an irreducible algebraic variety of dimension $d$, and let $E\subseteq V\times\mathbb A^m$ be an embedded algebraic vector bundle of rank $r>0$ and set $c:=\max\{d-(m-r),0\}.$ Then
\[
\Delta(E)>0
\quad\Longleftrightarrow\quad
\deg_{c,d+r-1-c}(\overline{E}')=0.
\]
\end{theorem}

\begin{proof}
We split the proof in cases:

If $r=m$, then $E=V\times\mathbb A^m$, so
$\Delta(E)=0$ and
$\deg_{d,m-1}(\overline E')=\deg(V)>0$. The case $d = 0$ is analogous.
We may therefore assume that $r<m$ and $d > 0$. 

If $d+r\leq m$, then $c=0$ and as we said before the assertion follows from
Lemma~\ref{observación positividad}, applied with $a=d$..

If $d+r>m$, then $c=d-(m-r)$. By Lemma~\ref{observación positividad}, applied with $a=m-r$, we have
\[
    \deg_{c,m-1}(\overline E')>0
    \quad\Longleftrightarrow\quad
    \dim(\mathscr{D}(E_{m-r}))=m.
\]

If $\Delta(E)>0$, then
\[
    \dim(\mathscr{D}(E_{m-r}))
    \leq\dim(\mathscr{D}(E))<m,
\]
so this bidegree vanishes.

Conversely, if $\Delta(E)=0$, then
$\dim\mathscr{D}(E)=m$. Since $d>0$ and $r<m$,
Proposition~\ref{prop:recovery-successive-sections}
applies and gives
\[
    \mathscr{D}(E_{m-r})=\mathscr{D}(E).
\]
Hence $\dim\mathscr{D}(E_{m-r})=m$, and the same
bidegree is positive.
\end{proof}

In particular, combining Theorem \ref{thm defect} with Proposition \ref{prop:recovery-successive-sections} we deduce the defect can be characterized in terms of the first non vanishing bidegree yielding a generalization of \cite[Theorem 1.1]{HOLME2001363} for general defects:

\begin{corollary}\label{cor:defect-bidegrees}
Let $V\subseteq\mathbb A^n$ be an irreducible algebraic variety of
dimension $d$, and let
$E\subseteq V\times\mathbb A^m$ be an embedded algebraic vector bundle
of rank $r>0$. Set $c:=\max\{d-(m-r),0\}.$ Then, for every $0 \leq k \leq \min\{d,m-r\}$ the following are equivalent:
\begin{enumerate}
    \item $\Delta(E)=k.$
    \item $\deg_{c+j,d+r-1-c-j}(\overline E')
=0,$ for every $0 \leq j < k $ and $\deg_{c+k,d+r-1-c-k}(\overline E')\neq 0.$
\qed\end{enumerate} 
\end{corollary}

\begin{remark}
  It is well known (see \cite{KOHN2021157}) that the bidegrees of the conormal bundle recover the polar degrees of a projective variety $X$. A classical theorem of Holme (see \cite[Theorem~1.1]{HOLME2001363}) characterizes the dual defect of $X$ in terms of the vanishing pattern of its polar degrees. Corollary~\ref{cor:defect-bidegrees} recovers Holme's theorem when $E$ is the conormal bundle and extends Holme's Theorem to other defects.
\end{remark}

\subsection{Grassmann Defectivity}\label{subsec:Gauss-bundle-duality}

\ \medskip

In this subsection, we study Grassmann defectivity and its
relation to directional defectivity
(see Theorem~\ref{thm:Gauss-annihilator-defects}). We start with a definition:

\begin{definition}
    Let \(V\subseteq\mathbb A^n\) be an irreducible algebraic variety of
dimension \(d\), and let
\(E\subseteq\mathbb A^{n+m}\) be an embedded algebraic vector bundle of
rank \(r<m\)  over \(V\).
We define the
\emph{annihilator bundle}
\[
E^\perp
:=
\left\{
(p,\lambda)\in V\times\mathbb A^m
\ \middle|\
\lambda(q)=0
\text{ for every }q\in L_p
\right\},
\]
where we think of $\lambda$ as an element of the dual space of $\mathbb{A}^m$. 
\end{definition}
By Observation~\ref{obs:localmente-trivial},
$E^\perp\subseteq\mathbb A^{n+m}$ is an embedded algebraic
vector bundle of rank $m-r$.
Moreover, taking annihilators twice recovers the original
bundle, yielding $(E^\perp)^\perp=E$.

\begin{example}
Let $X\subsetneq\mathbb P^n$ be a smooth projective variety
and set $V:=X\cap\{x_0\neq 0\}$.
The tangent cone bundle and the conormal bundle are annihilators of one another:
\[
    (\mathbb{T}V_{\mathrm{cone}})^\perp=N^*(V),
    \qquad
    N^*(V)^\perp=\mathbb{T}V_{\mathrm{cone}}.
\]
\end{example}

\medskip

Since $K$ has characteristic zero, the differential of the
classifying map $\Gamma_E:V\to\mathbb G(m,r)$ has rank
$\dim(\overline{\Gamma_E(V)})$ at a general point
$p\in V_{\mathrm{reg}}$. Thus Grassmann defectivity can
be computed as the dimension of the kernel of
\[
    d_p\Gamma_E:
    T_pV\longrightarrow T_{\Gamma_E(p)}\mathbb G(m,r).
\]

We are now ready to prove the main Theorem of this Subsection.

\begin{theorem}\label{thm:Gauss-annihilator-defects}
Let \(V\subseteq\mathbb A^n\) be a smooth irreducible algebraic variety of
dimension \(d\), and let
\(E\subseteq\mathbb A^{n+m}\) be an embedded algebraic vector bundle of
rank \(r<m\) over \(V\). Then
\[
\Delta_G(E)
=
\Delta_G(E^\perp)
\leq
\min\left\{
\Delta(E)+\max\{d+r-m,0\},
\Delta(E^\perp)+\max\{d-r,0\}
\right\}.
\]
In particular, \(E\) is Grassmann defective if and only if \(E^\perp\) is
Grassmann defective.
\end{theorem}

\begin{proof}
Taking orthogonal complements with respect to the standard pairing
defines an isomorphism
\[
\iota:
\mathbb G(m,r)
\xrightarrow{\ \sim\ }
\mathbb G(m,m-r),
\qquad
L\longmapsto L^\perp.
\]
The Grassmann direction maps of \(E\) and \(E^\perp\) are related by
\[
\Gamma_{E^\perp}
=
\iota\circ\Gamma_E.
\]
Since \(\iota\) is an isomorphism, the two maps have images of the same
dimension and, at every smooth point \(p\in V\), $\ker(d_p\Gamma_{E}) =\ker(d_p\Gamma_{E^\perp}).$
In particular,
\begin{equation}\label{eq:Gauss-annihilator}
\Delta_G(E^\perp)=\Delta_G(E).
\end{equation}

Now choose a general point \((p,e)\in E\), and set $F_p:=\Gamma_E^{-1}\bigl(\Gamma_E(p)\bigr)$.
For every \(q\in F_p\), one has \(L_q=L_p\). Since \(e\in L_p\), it
follows that \(e\in E_q\), and therefore
\[
F_p\times\{e\}\subseteq\pi_2^{-1}(e).
\]
By the fiber dimension theorem and the definition of \(\Delta(E)\), we obtain
\begin{align*} \Delta_G(E)
=
\dim F_p
\leq
\dim\pi_2^{-1}(e)
=
d+r-\dim\mathscr{D}(E)
&=
\Delta(E)
+d+r-\min\{d+r,m\}
\\ &=
\Delta(E)+\max\{d+r-m,0\}.
\end{align*}

Applying the same argument to \(E^\perp\), whose rank is \(m-r\), gives
\[
\Delta_G(E^\perp)
\leq
\Delta(E^\perp)+\max\{d-r,0\}.
\]

Combining these inequalities with
\eqref{eq:Gauss-annihilator} proves the first assertion.
\end{proof}

Zak's theorem on tangencies yields a bound for the Grassmann
defect of any embedded bundle containing the tangent cone
bundle, provided that the projective closure of the base
is smooth and non degenerate.

\begin{proposition}\label{cor:zak-grassmann}
Let $V\subseteq\mathbb A^n$ be an irreducible variety
of dimension $d$ whose projective closure
$X\subseteq\mathbb P^n$ is smooth and non degenerate.
Let $E\subseteq V\times\mathbb A^{n+1}$ be an embedded
algebraic vector bundle of rank $r$, with
$d+1\leq r\leq n$, such that $\mathbb{T}V_{\mathrm{cone}}\subseteq E.$
Then:  $$\Delta_G(E)=\Delta_G(E^\perp)\leq r-d-1.$$
\end{proposition}

\begin{proof}
Let $L\in\Gamma_E(V)$ be general and set
$Y:=\overline{\Gamma_E^{-1}(L)}\subseteq X$.
The inclusion $\mathbb{T}V_{\mathrm{cone}}\subseteq E$
implies $\mathbb T_pX\subseteq\mathbb P(L)$ along
$\Gamma_E^{-1}(L)$, and hence along $Y$, since the
Gauss map of $X$ is regular.
As $\dim\mathbb P(L)=r-1<n$, Zak's theorem on tangencies
(see \cite[Chapter~I, Corollary~1.8]{Zak1993}) gives
\[
    \Delta_G(E)=\dim Y
    \leq\dim\mathbb P(L)-\dim X=r-d-1.
\]
Finally, $\Delta_G(E)=\Delta_G(E^\perp)$ by Theorem \ref{thm:Gauss-annihilator-defects}
\end{proof}

For $E=\mathbb{T}V_{\mathrm{cone}}$, Proposition \ref{cor:zak-grassmann} is the usual application of Zak's theorem on tangencies that establishes
the finiteness of the Gauss map (see \cite[Chapter~I, Corollary~2.8]{Zak1993}).

\section{The Geometric Degree of an Embedded Algebraic Vector Bundle}\label{asdasdasd}

In this section, we derive a general formula for the geometric degree of an embedded algebraic vector bundle $E$ by means of van der Waerden's theorem on bidegrees (see Theorem \ref{teorema formula}). The key point is that each bidegree appearing in this formula admits a natural geometric interpretation in terms of the direction varieties of suitable linear restrictions of \(E\). This allows us to express the geometric degree of \(E\) as a sum of contributions determined by the geometry of these associated direction varieties. Finally in Subsection \ref{subsec trivial} we characterize algebraic vector bundles with minimal geometric degree.

\medskip

We start with an immediate consequence of
Theorem~\ref{thm:vdw}.

\begin{observation}\label{vdw}
Let $V\subseteq\mathbb A^n$ be an irreducible algebraic variety and let
$E\subseteq\mathbb A^{n+m}$ be an embedded algebraic vector bundle of positive rank
over $V$, then:
\begin{equation}\label{eq:vdw}
    \deg(E)
    =
    \sum_{\alpha+\beta=\dim V+\operatorname{rk}(E)-1}
    \deg_{\alpha,\beta}(\overline E').
\end{equation}
\end{observation}

Our first goal is to determine the bidegrees appearing in \eqref{eq:vdw}. We begin with the extremal cases, which can be computed directly from the geometry of the projection onto the base.

\begin{lemma}\label{lemacancelacion}
Let $V\subseteq\mathbb A^n$ be an irreducible algebraic variety of dimension
$d$, and let $E\subseteq\mathbb A^{n+m}$ be an embedded algebraic vector bundle of
positive rank $r$ over $V$. Let $\alpha,\beta\in\mathbb N_0$ satisfy
$\alpha+\beta=d+r-1$. Then:
\begin{enumerate}
    \item If $\alpha\geq d+1$, then
    $\deg_{\alpha,\beta}(\overline E')=0$.

    \item One has
    $\deg_{d,r-1}(\overline E')=\deg(V)$.
\end{enumerate}
\end{lemma}

\begin{proof}
Since $\pi_1(\overline E')\subseteq\overline V,$
a nonempty intersection computing $\deg_{\alpha,\beta}(\overline E')$ induces a
nonempty intersection of $\overline V$ with $\alpha$ general hyperplanes in
$\mathbb P^n$. This is impossible when $\alpha\geq d+1$, and proves the first
assertion.

For the second assertion, let $L\subseteq\mathbb P^n$ be a general linear
subspace of codimension $d$ such that
\[
    \overline V\cap L
\]
is contained in the affine chart $x_0\neq0$ and consists of $\deg(V)$
points. Over each $p\in\overline V\cap L$, the fiber of $\overline E'$ is
the projective linear space $\mathbb P(L_p)\simeq\mathbb P^{r-1}$. A
general linear subspace $H\subseteq\mathbb P^{m-1}$ of codimension $r-1$
meets each of these finitely many fibers in exactly one point. Consequently,
\[
   \deg_{d,r-1}(\overline E')= \#\bigl(\overline E'\cap(L\times H)\bigr)=\deg(V),
\]
and the result follows.
\end{proof}

We now interpret the remaining bidegrees in terms of the
direction varieties of general linear restrictions of $E$.
In our setting, this generalizes
\cite[Lemma 4.3]{CaminataCidRuizConca2023},
with both factors independent of the chosen general section
of a fixed dimension.
We give an independent proof, although the result also follows
by combining their lemma with Lemma \ref{lema tecnico}.

\begin{proposition}\label{producto}
Let $V\subseteq\mathbb A^n$ be an irreducible algebraic variety of dimension
$d$, and let $E\subseteq\mathbb A^{n+m}$ be an embedded algebraic vector bundle of
positive rank $r$ over $V$. For every $1\leq a\leq d$, there exists an
integer $\xi_{d-a+1}(E)\in\mathbb N_0$ such that
\[
    \deg_{d-a,a+r-1}(\overline E')
    =
    \xi_{d-a+1}(E)\deg\bigl(\mathscr{D}(E_a)\bigr).
\]
\end{proposition}

\begin{proof}
If $\dim\mathscr{D}(E_a)<a+r$, then
$\deg_{d-a,a+r-1}(\overline E')=0$ by
Corollary~\ref{observación positividad}. In this case we set
$\xi_{d-a+1}(E)=0$.

Assume now that $\dim\mathscr{D}(E_a)=a+r$. The projection
\[
    \pi_2\colon E_a\longrightarrow\mathscr{D}(E_a)
\]
is dominant between irreducible varieties of the same dimension. It is
therefore generically finite and becomes quasi-finite after restricting to
a suitable nonempty open subset of the target. By
Corollary~\ref{coro:fibra}, there exist a nonempty open subset
$U\subseteq\mathscr{D}(E_a)$ and an integer
$\xi_{d-a+1}(E)>0$ such that
\begin{equation}\label{eq:xi-general-fiber}
    \#\pi_2^{-1}(q)=\xi_{d-a+1}(E)
    \qquad\text{for every }q\in U.
\end{equation}
Notice that
\[
    \pi_2^{-1}(q)
    =
    \bigl\{(p,q)\in E_a\mid p\in V_a\bigr\}
    =
    \bigl\{(p,q)\in E_a\mid p\in V_a,\ q\in L_p\bigr\}.
\]
Thus~\eqref{eq:xi-general-fiber} simultaneously counts all the fibers
$L_p$ of $E_a\to V_a$ that contain the general direction $q$.

Set $k:=a+r$. Choose a general matrix
$G\in\mathbb A^{k\times m}$ and a general point $s\in\mathbb A^k$, and
write
\[
    H_{G,s}:=\{q\in\mathbb A^m\mid Gq=s\}.
\]
By Definition~\ref{def: grado geom}, the intersection
\[
    \mathscr{D}(E_a)\cap H_{G,s}
\]
consists of $\deg(\mathscr{D}(E_a))$ points. By choosing $(G,s)$ in a sufficiently
small nonempty open subset, we may assume that all these points are nonzero
and belong to $U$. It follows that
\begin{equation}\label{eq:affine-fiber-count}
\#\bigl(E_a\cap(\mathbb A^n\times H_{G,s})\bigr)
=
\sum_{q\in\mathscr{D}(E_a)\cap H_{G,s}}\#\pi_2^{-1}(q)\notag
=\xi_{d-a+1}(E)\deg\bigl(\mathscr{D}(E_a)\bigr).
\end{equation}

Consider the projective linear subspace
\[
    \widehat H_{G,s}
    :=
    \bigl\{[q]\in\mathbb P^{m-1}\ \bigm|\ Gq
    \text{ is proportional to }s\bigr\}.
\]
Since $G$ has rank $k$, the subspace $\widehat H_{G,s}$ has codimension
$k-1=a+r-1$. If $k=m$, then $\ker G=0$. If $k<m$, we may moreover assume
that
\[
    \pi_2(\overline E_a')\cap\mathbb P(\ker G)=\varnothing,
\]
because
\[
    \dim\pi_2(\overline E_a')+\dim\mathbb P(\ker G)
    =(k-1)+(m-k-1)=m-2<m-1.
\]
Hence the map $q\mapsto[q]$ induces a bijection
\[
    \mathscr{D}(E_a)\cap H_{G,s}
    \longrightarrow
    \pi_2(\overline E_a')\cap\widehat H_{G,s}.
\]
After shrinking the open set of choices once more, it can be assumed the corresponding points
of $\overline E_a'$ lie in the chart $\{x_0\neq0\}$. Therefore the same
rescaling gives a bijection between
\[
    E_a\cap(\mathbb A^n\times H_{G,s})
    \quad\text{and}\quad
    \overline E_a'\cap
    (\mathbb P^n\times\widehat H_{G,s}).
\]
Combining the preceding bijections with the affine intersection count, we
obtain,
\begin{align*}
\deg_{d-a,a+r-1}(\overline E')
=
\deg_{0,a+r-1}(\overline E_a')=
\xi_{d-a+1}(E)\deg\bigl(\mathscr{D}(E_a)\bigr),
\end{align*}
as desired. Finally, the bidegree on the left and the degree of
$\mathscr{D}(E_a)$ are independent of the general linear section defining $E_a$
by Corollary~\ref{observación del lema tecnico}. Hence
$\xi_{d-a+1}(E)$ depends only on $E$ and $a$.
\end{proof}

\begin{remark}\label{rmk: w(E)}
When $\dim\mathscr{D}(E_a)=a+r$, the preceding proof and
Corollary~\ref{coro:fibra} give
\[
    \xi_{d-a+1}(E)
    =
    \left[K(E_a):K\bigl(\mathscr{D}(E_a)\bigr)\right].
\]
Equivalently, for a general $q\in\mathscr{D}(E_a)$,
\[
    \xi_{d-a+1}(E)
    =
    \#\bigl\{p\in V_a\mid q\in L_p\bigr\}.
\]
 When $\dim\mathscr{D}(E_a)<a+r$, we have set
$\xi_{d-a+1}(E)=0$.
\end{remark}

Combining the preceding results with Proposition \ref{prop:recovery-successive-sections} yields the following Theorem

\begin{theorem}[Theorem B from the introduction]\label{teorema formula}
Let $V\subseteq\mathbb A^n$ be an irreducible algebraic variety of dimension
$d$, and let $E\subseteq\mathbb A^{n+m}$ be an embedded algebraic vector bundle of
positive rank $r$ over $V$. Set $s:=\min\{d,m-r\}-\Delta(E),$ then:
\[
\deg(E)
=
\deg(V)
+
\xi_{d-s+1}(E)\deg\bigl(\mathscr D(E)\bigr)
+
\sum_{a=1}^{s-1}
\xi_{d-a+1}(E)\deg\bigl(\mathscr D(E_a)\bigr),
\]
where we set $\xi_{d+1}(E)=0$ in case $s=0$.
\end{theorem}

\begin{proof}
If $s = 0$, Lemma \ref{lema 7} implies $E$ is trivially embedded and the result follows. By Observation~\ref{vdw},
\[
    \deg(E)
    =
    \sum_{\alpha+\beta=d+r-1}\deg_{\alpha,\beta}(\overline E').
\]
Lemma~\ref{lemacancelacion} shows that the terms with $\alpha\geq d+1$ vanish
and that the term with $\alpha=d$ equals $\deg(V)$. For each remaining term,
write $\alpha=d-a$, with $1\leq a\leq d$, and apply
Proposition~\ref{producto}. The result follows immediately from Proposition \ref{prop:recovery-successive-sections}.
\end{proof}

In particular, in our context, Huh's Theorem states the following:
\begin{proposition}\label{proposition huh}
    Let $V \subseteq \mathbb{A}^n$ be an irreducible algebraic variety of dimension $d$ and $E$ an embedded algebraic vector bundle of rank $r$ over $V$ and $1 \leq a \leq d$. Denote $s:=\min\{d,m-r\}-\Delta(E),$ $b_0 = \deg(V)$ and $b_a = \deg_{d-a,a+r-1}(\overline E')  =
    \xi_{d-a+1}(E)\deg\bigl(\mathscr{D}(E_a)\bigr) $ for $1 \leq a \leq d$. Then we have that the sequence $(b_j)_{j=0}^d$ is log-concave with no internal zeros: 
    $$b_a^2 \geq b_{a+1} b_{a-1}, \qquad b_i \neq 0 \ \iff \ 0 \leq i \leq s$$
\end{proposition}
\begin{proof}
    See \cite[Theorem 21]{Huh2012}.
\end{proof}

 Notice we can iterate these estimates and for every $1 \leq a \leq s:= \min\{d,m-r\}-\Delta(E)$ obtain:
\begin{equation}\label{ecuacion huh}
    \xi_{d-a+1}(E)\deg(\mathscr{D}(E_a)) \leq \frac{\deg(\mathscr{D}(E_1))^a \xi_{d}(E)^a}{\deg(V)^{a-1}}.
\end{equation}
So, by Theorem \ref{teorema formula} a full general bound for $\deg(E)$ can be obtained in terms of $\deg(\mathscr{D}(E_1)) \xi_d(E)$ and $\deg(V)$:
\begin{equation*}
    \deg(E) \leq \deg(V)\sum_{i=0}^s \left ( \frac{\deg(\mathscr{D}(E_1))\xi_{d}(E)}{\deg(V)} \right)^i.
\end{equation*}

\subsection{Embedded algebraic vector bundles of minimal degree}\label{subsec trivial} \ \medskip

Recall from Example~\ref{ex:trivially-embedded-bundle} that an embedded
algebraic vector bundle
\[
E\subseteq\mathbb A^{n+m}
\]
over \(V\subseteq\mathbb A^n\) is said to be \emph{trivially embedded}
if $E=V\times L$
for some linear subspace \(L\subseteq\mathbb A^m\). Equivalently, all
the fibers of \(E\) coincide with the same linear subspace \(L\).

We first characterize embedded triviality by maximal directional
defect and the vanishing of a single bidegree. We then relate
these conditions to the minimality of the degree of the total space.

\begin{proposition}\label{proposición cancelación}
Let \(V\subseteq\mathbb A^n\) be an irreducible algebraic variety of
dimension \(d\geq1\), and let
\(E\subseteq\mathbb A^{n+m}\) be an embedded algebraic vector bundle
of rank \(r\) over \(V\). The following conditions are equivalent:
\begin{enumerate}
    \item \(E\) is trivially embedded.
    
    \item \label{(2)} $\Delta(E) = \min\{d,m-r \}$
    
    \item $\deg_{d-1,r}(\overline E')=0$
\end{enumerate}
\end{proposition}
    \begin{proof}
Conditions~(1) and~(2) are equivalent by Lemma~\ref{lema 7},
and $(1)\Rightarrow(3)$ follows from
Lemma~\ref{observación positividad}. Finally, condition $(3)$ implies $(1)$ by Proposition \ref{proposition huh}, as the bidegree sequence has no internal zeros.
\end{proof}

Notice that \ref{teorema formula} shows that for an embedded algebraic vector bundle $E$ we have the inequality:
$$\deg(E) \geq \deg(V).$$
The next result is an extension of \cite[Theorem 35]{jeronimo2025geometric} to the setting of embedded algebraic vector bundles. It characterizes the embedded algebraic vector bundles which have this minimal degree. 

\begin{corollary}\label{coro: triviales}
Let \(V\subseteq\mathbb A^n\) be an irreducible algebraic variety and
let \(E\subseteq\mathbb A^{n+m}\) be an embedded algebraic vector
bundle over \(V\). Then
\[
E\text{ is trivially embedded}
\quad\Longleftrightarrow\quad
\deg(E)=\deg(V).\] 
\end{corollary}
\begin{proof}
    The result follows immediately from Condition \ref{(2)} in Proposition \ref{proposición cancelación} and Theorem \ref{teorema formula}.
\end{proof}

\section{Multiplicities and Tangent Bundle Degrees}

In this section, we apply the theory developed in
Sections \ref{asdff} and \ref{asdasdasd} to classical constructions.
The degree formula in Theorem \ref{teorema formula} plays a
central role in proving Theorem \ref{teorema cuadrado}
(Theorem C in the introduction), which gives a positive answer
to \cite[Question 33]{jeronimo2025geometric}.

\medskip

The section is organized as follows.
Subsection \ref{multiplicities} relates the multiplicities
$\xi_i$ of the tangent bundle $TV$, the secant bundle
$\mathcal{S}(V)$, and the tangent cone bundle
$\mathbb{T}V_{\operatorname{cone}}$ to classical enumerative
invariants (see Remark \ref{remark seto},
Proposition \ref{prop:secant-multiplicities}, and
Observation \ref{ejemplo TV}).
Subsection \ref{deg estiamtes} establishes bounds for the
geometric degree of tangent bundles of smooth algebraic
varieties, culminating in Theorem \ref{teorema cuadrado}.

\subsection{Multiplicities and enumerative invariants}\label{multiplicities}
\ \medskip

The multiplicities introduced in Proposition~\ref{producto} recover
classical enumerative invariants in the tangent and secant cases.
For instance, for a smooth affine curve $\CC \subseteq \mathbb{A}^n$,
the invariant $\omega(\CC)$ introduced in
\cite[Definition 17]{jeronimo2025geometric} counts the points
$p \in \CC$ satisfying $T_p\CC = T_q\CC$ for a general point
$q \in \CC$. The following observation is therefore immediate
from the definitions.

\begin{observation}\label{ejemplo TV}
Let $\CC \subseteq \mathbb{A}^n$ be a smooth irreducible affine curve.
Then
$$
\omega(\CC) = \xi_1(T\CC).
\qed
$$
\end{observation}

A further example is provided by the secant degree, which we
recall below.

\begin{definition}\label{def:secantdegree}
Let \(X\subseteq\mathbb P^n\) be an irreducible projective variety of
dimension \(d\). The \emph{secant degree}
        of \(X\), denoted by \(\mu(X)\), is the number of secant lines to \(X\) passing through a general point of \(\operatorname{Sec}(X)\) if $\dim(\operatorname{Sec}(X))=2d+1$ and $0$ otherwise.
\end{definition}

\begin{observation}
\label{prop:secant-multiplicities}
Let $V \subseteq \mathbb{A}^n$ be an irreducible algebraic variety of dimension $d>0$ and \(X\subseteq\mathbb P^n\) be its projective closure.  Assume that $ 2d \leq n$. Then:
\[
    \xi_1(\mathcal{S}(V))
    =
    2\mu(X) \]

\end{observation}

\begin{proof}
Let \(\pi_{\mathrm{sec}}\) be the projection of \(\mathcal{S}(V)\) onto
the direction coordinates. If
\(\dim\operatorname{Sec}(X)=2d+1\), then \(\pi_{\mathrm{sec}}\) is
generically finite. Since its general fiber parametrizes ordered pairs
defining secant lines, Corollary~\ref{coro:fibra} and
Remark~\ref{rmk: w(E)} give
\[
    \xi_1(\mathcal{S}(V))
    =
    \#\pi_{\mathrm{sec}}^{-1}(q)
    =
    2\mu(X).
\]
If $\dim(\operatorname{Sec}(X)) < 2d+1$ then $\mu(X) =0$ and also $\xi_1(\mathcal{S}(V)) = 0$
\end{proof}

We continue with a definition due to Severi
(see \cite{Severi1902} or \cite{HernandezGomezRusso2026}).

\begin{definition}\label{ceto}
Let $X\subseteq\mathbb P^n$ be an irreducible projective
variety of dimension $d$, and let $0\leq i\leq d$.
If $d+i\leq n$, choose a general linear subspace
$L\subseteq\mathbb P^n$ of codimension $d-i$ and a general
linear subspace $M\subseteq L$ of codimension $2i$ in $L$.
The \emph{$i$-th ceto} of $X$ is
$$
\omega_i(X):=
\#\bigl\{x\in(X\cap L)_{\mathrm{reg}}:
\mathbb T_x(X\cap L)\cap M\neq\varnothing\bigr\}.
$$
If $d+i>n$, we set $\omega_i(X)=0$.
\end{definition}

Now we will relate the Ceti with the bidegrees of $\mathbb{T}V_{\operatorname{cone}}.$

\begin{proposition}\label{proposetomultig}
Let \(X\subseteq\mathbb P^n\) be a smooth irreducible projective variety of dimension \(d\), let \(V=X\cap\{x_0\neq0\}\), and consider the embedded algebraic vector bundle \(\mathbb{T}V_{\mathrm{cone}}\). Then, for every \(0\leq i\leq d\),
\[
    \omega_i(X)
    =
    \deg_{d-i,d+i}
    \bigl((\overline{\mathbb{T}V_{\mathrm{cone}}})'\bigr).
\]
\end{proposition}

\begin{proof}
If $d+i>n$, both sides vanish by convention.
Assume therefore that $d+i\leq n$.
Since $X$ is smooth and $V$ is dense in $X$, one has
$$
(\overline{\mathbb TV_{\mathrm{cone}}})'
=
\left\{
(x,[y])\in X\times\mathbb P^n:
[y]\in\mathbb T_xX
\right\}.
$$

Let $L\subseteq\mathbb P^n$ be a general linear subspace
of codimension $d-i$, and set $Y=X\cap L$.
Then $Y$ is smooth of dimension $i$, and
$$
\mathbb T_xY=\mathbb T_xX\cap L
\qquad\text{for every }x\in Y.
$$
Consequently,
$$
(\overline{\mathbb TV_{\mathrm{cone}}})'\cap(L\times L)
=
\left\{
(x,[y])\in Y\times L:
[y]\in\mathbb T_xY
\right\}.
$$
The equations of $L$ thus have constant rank
$(d+1)-(i+1)=d-i$ on the corresponding tangent vector
spaces. Since $Y$ is smooth, the incidence variety
on the right is smooth of dimension $2i$.

Choose a general linear subspace $M\subseteq L$
of codimension $2i$ in $L$.
Its defining equations cut this incidence variety
in finitely many simple points, or give an empty
intersection.
Moreover, each nonempty intersection
$\mathbb T_xY\cap M$ is linear and finite, hence consists
of exactly one point. By definition of the $i$-th ceto,
$$
\omega_i(X)
=
\#\bigl(
(\overline{\mathbb TV_{\mathrm{cone}}})'
\cap(L\times M)
\bigr).
$$

Finally, $L$ and $M$ have codimensions $d-i$ and $d+i$
in $\mathbb P^n$, respectively.
The intersection is proper and all its points are simple,
so the multiprojective Bézout theorem
\cite[Theorem~1.11]{D2013} gives
$$
\omega_i(X)
=
\deg_{d-i,d+i}
\bigl((\overline{\mathbb TV_{\mathrm{cone}}})'\bigr).
\qedhere$$
\end{proof}

\begin{example}\label{remark seto}
Let $V \subseteq \mathbb{A}^n$ be a smooth irreducible algebraic variety of dimension $d$ with smooth projective closure $X$ and suppose \(n\geq2d\). Taking \(a=d\) in
Proposition~\ref{producto} gives
\begin{equation}\label{ecuacion seto}
\begin{aligned}
\omega_d(X)
=
\deg_{0,2d}\bigl((\overline{\mathbb{T}V_{\mathrm{cone}}})'\bigr) &=
\xi_1(\mathbb{T}V_{\mathrm{cone}})
\deg\bigl(\mathscr{D}(\mathbb{T}V_{\mathrm{cone}})\bigr)\\
&=
\xi_1(\mathbb{T}V_{\mathrm{cone}})\deg\bigl(\Tan(X)\bigr).
\end{aligned}
\end{equation}
Thus \(\xi_1(\mathbb{T}V_{\mathrm{cone}})\) is the tangent degree $\tau(X)$ studied in \cite{HernandezGomezRusso2026}. 
\end{example}

Determining the possible values of the multiplicities is an
interesting problem in its own right. For instance in
\cite[Theorem 10]{HernandezGomezRusso2026}, the authors establish the restriction
\[
\xi_1(\mathbb{T}V_{\operatorname{cone}})\neq1
\]
for a non degenerate $d$-dimensional projective variety
$X\subseteq\mathbb P^{2d}$. For conormal bundles, the story is different. The bidegrees recover polar degrees
(see \cite{KOHN2021157}), and all nonzero multiplicities
$\xi_i(N^*(V))$ equal $1$ \cite{Tevelev2003}.

\begin{remark}\label{remark approach}Jorgenson uses polar-degree vanishing and numerical bounds
to establish numerous cases of the duality defect conjecture
(see \cite[Algorithm~1, Theorem 3.3 and Theorem~4.7]{https://doi.org/10.1112/blms.12379}).
We believe the effective defectivity criterion (see Theorem~\ref{thm defect}),
combined with the log-concavity constraints (see Proposition~\ref{proposition huh}) and their resulting bounds~\eqref{ecuacion huh}, together with explicit bundle equations, provide analogous tools for extending this
computational strategy to other defectivity conjectures such as the Abo-Ottaviani-Peterson conjecture on secant varieties of Segre
varieties (see \cite{AboOttavianiPeterson2009}) and the
Baur-Draisma-de Graaf conjecture on secant varieties of
Grassmannians (see \cite{BaurDraismaDeGraaf2007}).
\end{remark}

\subsection{The Tangent Bundle}\label{deg estiamtes}  \ \medskip

We now specialize the preceding theory to its motivating example: the
tangent bundle of a smooth affine algebraic variety. The geometric
degree of this bundle and its relation with the corresponding variety
of tangent directions were studied in
\cite{jeronimo2025geometric}. Let $s:= \min\{d,n-d\}-\Delta(TV)$. Applying
Theorem~\ref{teorema formula} to \(E=TV\), we obtain
\begin{equation}\label{eq:tangent-bundle-degree-formula}
\deg(TV)
=
\deg(V)
+\xi_{d-s+1}(TV) \deg(\operatorname{Tan}_0(V)) +
\sum_{a=1}^{s-1}
\xi_{d-a+1}(TV)
\deg\bigl(\mathscr{D}((TV)_a)\bigr).
\end{equation}
This in particular recovers \cite[Theorem~21]{jeronimo2025geometric} when $V$ is a curve.

\medskip 
 In \cite[Theorem~30]{jeronimo2025geometric}, is proven that:
\[
\deg(TV)
\leq
\min\left\{
\deg(V)^{\,n-d+1},
\,
\deg(V)\bigl((n-d)(\deg(V)-1)+1\bigr)^d
\right\}.
\]
Although intrinsic, these bounds depend essentially on the dimension
and codimension of \(V\), and their degree as polynomials in $\deg(V)$
increases with these parameters. Motivated by examples (see \cite[Corollary 32]{jeronimo2025geometric}), this led to
\cite[Question~33]{jeronimo2025geometric}, asking whether
\(\deg(TV)\) admits a universal upper bound which is quadratic in
\(\deg(V)\). In this subsection we will give a positive answer to this question (see Theorem \ref{teorema cuadrado} below). We start with an Proposition:

\begin{proposition}\label{thm}
Let \(V\subseteq\mathbb A^n\) be a smooth irreducible affine variety of
dimension \(d\), and suppose that its projective closure
\(X\subseteq\mathbb P^n\) is smooth. Then 
$$\deg(\mathbb{T}V_{\operatorname{cone}})=\begin{cases} \deg(V)^2- 2\mu(X)\deg(\operatorname{Sec}(X)) & \text{if } 2d \leq n. \\
\deg(V)^2 & \text{if } 2d>n.
\end{cases}$$
\end{proposition}

\begin{proof}
If $2d\leq n$ notice that:
\begin{equation*}
\begin{aligned}
2\mu(X) \deg(\operatorname{Sec}(X))
&=
\deg(X)\bigl(\deg(X)-1\bigr)
-\sum_{i=1}^d\omega_i(X)
\\&=
\deg(X)\bigl(\deg(X)-1\bigr)
-\sum_{i=1}^d
\deg_{d-i,d+i}
\bigl((\overline{\mathbb{T}V_{\mathrm{cone}}})'\bigr)
=\deg(X)^2-\deg(\mathbb{T}V_{\mathrm{cone}}),
\end{aligned}
\end{equation*}
where the first equality is Severi's Double Point Formula
(see \cite{Severi1902}), the second follows from
Proposition~\ref{proposetomultig}, and the last one follows from
Observation~\ref{vdw}.

For the other case, suppose that \(2d>n\), and set \(c=n-d\).
Then \(2c<n\). Let $L\simeq\mathbb P^{2c}\subseteq\mathbb P^n$
be a general linear subspace and set \(Y:=X\cap L\). Since \(L\) is general and has
codimension \(2d-n=d-c\), by Bertini's Theorem (see \cite{jouanolou1983}) the variety \(Y\) is smooth of dimension \(c\)
and:
\[
    \deg(Y)=\deg(X)=\deg(V).
\]
Moreover, for \(1\leq i\leq c\), a general \(i\)-dimensional linear
section of \(Y\) is also a general \(i\)-dimensional linear section of
\(X\). Hence
\begin{equation*}\label{setata}
    \omega_i(Y)=\omega_i(X),
    \qquad 1\leq i\leq c.
\end{equation*}
The result then follows by a similar reasoning to the first case as the left part of Severi's double point formula vanishes for $Y \subseteq \mathbb{P}^{2c}$ and the remaining ceti vanish.
\end{proof}

With this we are ready to give a proof of the main Theorem of the section.

\begin{theorem}[Theorem C from the introduction]\label{teorema cuadrado}
Let $V\subseteq\mathbb A^n$ be a smooth irreducible affine
variety, with its tangent bundle naturally embedded as
$TV\subseteq\mathbb A^{2n}$ and let $X \subseteq \mathbb{P}^n$ be its projective closure. Then
$$\deg(TV)\leq \deg(V)^2.$$
Moreover: \begin{enumerate}
    \item For every $d,D \in \mathbb{N}$ such that $d+1 \leq n$ there exists an irreducible smooth affine variety $V_{D,d}$ with $\deg(V_{D,d})=D$ and $\dim(V_{D,d}) = d$ such that the previous inequality is an equality.
    \item If $X$ is smooth and $2\dim(V) \leq n$, then the stronger bound holds:
    $$\deg(TV) \leq \deg(V)^2-2\mu(X)\deg(\operatorname{Sec}(X)).$$
\end{enumerate}
\end{theorem}
\begin{proof}

Fix $d=\dim(V)$ and an arbitrary point $p\in V$. Since $V$ is smooth, it is locally a complete intersection around $p$. Hence there exist a nonempty principal Zariski open subset $U:=\{h\neq0\}\subseteq V$ and polynomials $f_1,\ldots,f_{n-d}\in I(V)$ such that $p\in U$, the ideal $I(V)\mathcal O_q\subseteq\mathcal O_q$ is generated by $f_1,\ldots,f_{n-d}$ for every $q\in U$, and
\begin{equation*}
    TU = \{(\mathbf{x},\mathbf{y}) \in \mathbb{A}^{2n} \mid h(\mathbf{x}) \neq 0,  \text{ and } f_i(\mathbf{x})= \nabla f_i \cdot \mathbf{y} = 0, \quad 1 \leq i \leq n-d\}.
\end{equation*}

Define the incidence variety $\mathcal I\subseteq\mathbb A^{2n}\times\mathbb A^1$ by
\[
\mathcal I :=
\left\{
(\mathbf{x},\mathbf{y},t)\in\mathbb{A}^{2n}\times\mathbb{A}^1
\;\middle|\;
\begin{array}{@{}l@{}}
    h(\mathbf{x})\cdot h(\mathbf{x}+t\mathbf{y})\neq 0
    \\[2mm]
    \displaystyle
    f_i(\mathbf{x})
    =
    G_i(\mathbf{x},\mathbf{y},t)
    =
    \frac{f_i(\mathbf{x}+t\mathbf{y})-f_i(\mathbf{x})}{t}
    =0,
    \quad 1\leq i\leq n-d
\end{array}
\right\}.
\]
Here $G_i\in K[\mathbf{x},\mathbf{y},t]$, because expansion of the numerator shows that every term is divisible by $t$. Let
$$\vartheta:\mathcal{I} \rightarrow \mathbb{A}^1, \qquad \vartheta(\mathbf{x},\mathbf{y},t) = t.$$

We first study its fibers. For $t=0$, the fiber consists of the points $(\mathbf{x},\mathbf{y},0)\in\mathcal I$ satisfying
\[
h(\mathbf{x})\neq 0,\qquad
f_i(\mathbf{x})=0,\qquad
\nabla f_i(\mathbf{x})\cdot\mathbf{y}=0,
\]
because $G_i(\mathbf{x},\mathbf{y},0)=\nabla f_i(\mathbf{x})\cdot\mathbf{y}$. Consequently,
$$\vartheta^{-1}(0)=TU\times\{0\}.$$

If $t\neq0$, the equalities $f_i(\mathbf{x})=G_i(\mathbf{x},\mathbf{y},t)=0$ are equivalent to $f_i(\mathbf{x})=f_i(\mathbf{x}+t\mathbf{y})=0$. The open condition also guarantees that both $\mathbf{x}$ and $\mathbf{x}+t\mathbf{y}$ belong to $U$.

Consider the linear morphism
\[
\Phi_t:\vartheta^{-1}(t)\longrightarrow U\times U,
\qquad
(\mathbf{x},\mathbf{y},t)
\longmapsto
(\mathbf{x},\mathbf{x}+t\mathbf{y})
\]
The map $\Phi_t$ is linear, with linear inverse $\displaystyle (\mathbf{x},\mathbf{z})
\longmapsto
\left(
\mathbf{x},
\frac{\mathbf{z}-\mathbf{x}}{t},
t
\right).$
Thus all fibers of $\vartheta$ are irreducible of dimension $2d$, and for $t\neq0$,
$\deg(\vartheta^{-1}(t))=\deg(U\times U\times\{t\})$.

Allowing $t$ to vary in the preceding maps gives an isomorphism
$$\mathcal{I}\cap \{t \neq 0\} \simeq U \times U \times\left ( \mathbb{A}^{1} \setminus \{0\} \right ).$$
Hence $\mathcal{I}\cap\{t\neq0\}$ is a nonempty irreducible open subset of $\mathcal I$ of dimension $2d+1$. Its closure in $\mathcal I$ is therefore an irreducible component. Any other component would have to lie in the hyperplane $\{t=0\}$. But
$$\mathcal{I}\cap\{t=0\}=\vartheta^{-1}(0)=TU\times\{0\},$$
so such a component would have dimension $2d$. This is impossible because $\mathcal I$ is defined by $2(n-d)$ equations, and every component must therefore have dimension at least $2d+1$. Thus $\mathcal I$ is irreducible.

Applying Lemma \ref{Lema defomracion} to the irreducible quasi-affine variety $\mathcal I$ and the map $\vartheta$, we obtain for general $t\in\mathbb A^1\setminus\{0\}$:
\begin{align*}
    \deg(TV) = \deg(TU \times \{0\}) = \deg(\vartheta^{-1}(0))  \leq  \deg(\vartheta^{-1}(t)) & = \deg(U \times U \times \{t\}) \\ & = \deg(U)^2 = \deg(V)^2,
\end{align*}
where $\deg(U \times U \times \{0\}) = \deg(U)^2$ follows from \cite[Proposition 2]{Heintz1983}. This proves the main assertion. Finally settle the remaining conditions independently:
\begin{enumerate}
    \item By \cite[Example~27]{jeronimo2025geometric}, there exists a smooth
irreducible plane curve \(C_r\subseteq\mathbb A^2\) such that
\[
\deg(C_r)=r
\qquad\text{and}\qquad
\deg(TC_r)=r^2.
\]
One may take \(C_r=Z(x_1^r+x_2^r-1)\), since this curve is linearly
isomorphic to those considered in that example. Set $\widetilde V(d,r):=C_r\times\mathbb A^{d-1}
   \subseteq\mathbb A^{d+1}.$
Then \(\widetilde V(d,r)\) is smooth and irreducible of dimension \(d\),
and
\[
T\widetilde V(d,r)=TC_r\times\mathbb A^{2d-2}.
\]
Since taking a product with an affine space does not change geometric
degree,
\[
\deg\widetilde V(d,r)=r,
\qquad
\deg\bigl(T\widetilde V(d,r)\bigr)=r^2.
\]
Finally, if \(d+1<n\), we embed \(\mathbb A^{d+1}\) linearly into
\(\mathbb A^n\). This preserves both degrees and gives the required
variety \(V(d,r)\). 

\item Let \(H\subseteq\mathbb A^{n+1}\) be the hyperplane defined by
\(\lambda=0\). By construction,
\[
\mathbb{T}V_{\mathrm{cone}}\cap(V\times H)
=
\left\{
\bigl(p,(0,v)\bigr)
\ \middle|\
p\in V,\ v\in T_pV
\right\}
\simeq TV.
\]
Thus \(TV\) is linearly isomorphic to a hyperplane section of
\(\mathbb{T}V_{\mathrm{cone}}\), and Bézout's inequality gives
\[
    \deg(TV)\leq\deg(\mathbb{T}V_{\mathrm{cone}}),
\]
and the result follows straightforwardly from Proposition \ref{thm}. \qedhere
\end{enumerate}
\end{proof}
\begin{remark}
The proof of the quadratic bound in the general case relies
on the deformation argument of Lemma~\ref{Lema defomracion}.
We expect that an alternative proof could be obtained by
extending Proposition~\ref{proposetomultig} to varieties with
singular projective closure and applying a suitable singular
version of Severi's double point formula
(see~\cite{Severi1902,FultonLaksov1977,Catanese1979,CataneseOguiso2020}).
Such an approach could also yield sharper bounds through
nonnegative correction terms accounting for the singularities
of the projective closure.
\end{remark}

\begin{remark}
    The deformation considered in the proof of Theorem \ref{teorema cuadrado} is inspired by the deformation to the normal cone. This produces a family whose general fiber is isomorphic to a fixed variety and whose special fiber is the corresponding normal cone (see \cite[Chapter~4, \S4.1]{fulton1998intersection}). Applied to the diagonal of a smooth variety, the normal cone identifies with its tangent bundle.
\end{remark}

We now proceed to deduce a Corollary for the degree of the secant and tangential varieties:

\begin{corollary}\label{cor:tangential-degree-bounds}
    Let $X \subseteq \mathbb{P}^n$ be a smooth non degenerate projective variety of dimension $d$ and $V:= X\cap \{x_0 \neq 0\} \subseteq \mathbb{A}^n$. Let $s:=\min\{\dim(V),\operatorname{codim}(V)\}-\Delta(\mathbb{T}V_{\operatorname{cone}})$. 
    \begin{enumerate}
        \item If $X$ is secant defective then:
        $$\deg(\operatorname{Sec}(X)) = \deg(\operatorname{Tan}(X)) \leq \frac{\deg(X)(\deg(X)-1)}{\xi_{d-s+1}(\mathbb{T}V_{\operatorname{cone}})}.$$
        \item If $X$ is not secant defective and $2\dim(V) \leq n$:
         $$2\mu(X) \deg(\operatorname{Sec}(X)) + \xi_{d-s+1}(\mathbb{T}V_{\operatorname{cone}}) \deg(\operatorname{Tan}(X)) \leq \deg(X)(\deg(X)-1).  $$
    \end{enumerate}
\end{corollary}
\begin{proof}
    The result follows combining Theorems \ref{teorema formula} and \ref{Teorema Zak} with Proposition \ref{thm}.
\end{proof}

Combining Theorem \ref{teorema cuadrado} with Theorem \ref{teorema formula} in the special case $d=1$ and $E = TV$ yields a sharper version of \cite[Proposition 18]{jeronimo2025geometric} for curves, giving a geometric bound on coincident tangent spaces.

\begin{corollary}\label{corolario: curvas estiamcion omega}
    Let $\CC \subseteq \mathbb{A}^n$ be a smooth irreducible algebraic curve. Then
$$\omega(\CC) \deg(\operatorname{Tan}_0(\CC)) \leq \deg(\CC)(\deg(\CC)-1).\qed$$ 
\end{corollary}
\section{Applications to Prolongation Varieties}

Let \(K\) be an algebraically closed field of characteristic zero, as before, but now endow \(K\) with a derivation \(\partial\). For instance, one may consider \(\mathbb{C}\) equipped with the trivial derivation \(\partial\equiv 0\). We will then say that \(K\) is a \emph{differential field}. For a differential field \(K\), one can define the ring of differential polynomials
\[
K\{\mathbf{x}\}:=K[\mathbf{x}^{(i)}\mid i\in\mathbb{N}_0],
\]
where each derivative of a differential variable \(\mathbf{x}=(x_1,\ldots,x_n)\) is regarded as a new variable in the polynomial ring. The derivation \(\partial\) naturally extends to \(K\{\mathbf{x}\}\) by setting
\[
\partial(x_i^{(j)})=x_i^{(j+1)},
\]
and, in particular, \(\partial(x_i)=x_i^{(1)}=:x_i'\). We will use \(\mathbf{x}^{(\leq k)}\) to denote the variables \(\mathbf{x}\) together with all their derivatives up to order \(k\).

\medskip

A \emph{differential ideal} is an ideal of the ring \(K\{\mathbf{x}\}\) that is closed under derivation. For a subset \(J\subseteq K\{\mathbf{x}\}\), we denote by \([J]\) the smallest differential ideal containing \(J\). Notice that a system of algebraic differential equations, that is, a system of differential equations whose equations can be written as polynomials
\[
f_i(\mathbf{x}^{(\leq k)}),
\]
can naturally be encoded by the differential ideal of \(K\{\mathbf{x}\}\) generated by these polynomials. For more background on differential algebra, we refer the reader to \cite{ritt1950differential,kolchin1973differential}.

\subsection{Prolongations and Tangent Bundles} \ \medskip

We begin by defining the prolongation variety of a system of algebraic differential equations of order $k$. 

\begin{definition}\label{definicion prolongacion}
    Let $r,k \in \mathbb{N}$ and let $\Sigma:= \{f_1(\mathbf{x}^{(\leq k)})=0,\ldots,f_r(\mathbf{x}^{(\leq k)})=0\}$ be a system of algebraic differential equations, where each $f_i\in K\{\mathbf{x}\}$ has order at most $k$.
    \begin{itemize}
        \item We define the \emph{prolongation of $\Sigma$} as the system of algebraic differential equations
    \begin{equation*}
        \operatorname{Prol}(\Sigma):= \left  \{f_i(\mathbf{x}^{(\leq k)})= f_i^\partial+ \sum_{j=0}^k \nabla_{\mathbf{x}^{(j)}} f_i \cdot \mathbf{x}^{(j+1)} = 0 \ \forall 1 \leq i \leq r \right \},
    \end{equation*}
where, $f_i^\partial$ is simply applying the derivation $\partial$ to the coefficients of $f$ and $\nabla_{\mathbf{x}^{(j)}} f_i \cdot \mathbf{x}^{(j+1)}$ denotes the usual Euclidean inner product between the gradient of $f_i$ with respect to $\mathbf{x}^{(j)}$ and the vector of derivatives $\mathbf{x}^{(j+1)}$.
        \item The \emph{prolongation variety of $\Sigma$} is the algebraic variety $Z(\operatorname{Prol}(\Sigma)) \subseteq \mathbb{A}^{n \times(k+2)}$.
    \end{itemize}
\end{definition}

The systems $[\Sigma]$ and $[\operatorname{Prol}(\Sigma)]$ generate the same differential ideal and therefore have the same solution sets as systems of differential equations.

\begin{example}
  Prolongation varieties are useful for describing the constraints that every solution, including its initial conditions, must satisfy. To illustrate this, we reproduce a simple example from \cite[p.~11]{AscherPetzold1998}. Take $K$ as the algebraic closure of $\mathbb{C}(t)$ with $\partial:= \frac{d}{dt}$. Consider the system
$$
\Sigma:=\begin{cases}
x_1'-x_2=0,\\
x_1-t=0
\end{cases}
.$$
Prolonging both equations once gives
$$
\operatorname{Prol}(\Sigma):=\begin{cases}
x_1'-x_2=0,& \ x_1''-x_2'=0,\\
x_1-t=0, & \ x_1'-1=0.
\end{cases}
$$
Combining the equation $(x_1'-x_2=0)$ with the prolonged equation $(x_1'-1=0)$ yields the additional constraint $(x_2=1)$. Thus, at $t=0$, every solution must satisfy the initial conditions
$$x_1(0)=0
\qquad\text{and}\qquad
x_2(0)=1.$$  
\end{example}

\begin{remark}
    For an ODE system, that is, a system of the form
\[
\mathbf{x}'=\mathbf{f}(\mathbf{x}) \quad \text{with } n=r,
\]
existence and uniqueness theorems are available for solutions with a prescribed initial condition. This is usually known as the Cauchy principle (see \cite{JMPA_1893_4_9__217_0}).

For a system of algebraic differential equations, by contrast, the existence of solutions is less transparent, and some systems do not admit solutions. At least locally, the issue is closely tied to the implicit function theorem: one would like to solve for the highest-order derivatives in terms of the remaining variables and then apply a suitably adapted Cauchy principle. It is easy to exhibit DAE systems that cannot initially be treated in this way but acquire the required form after a finite number of prolongations. This idea appears in Cartan's pioneering work \cite{cartan1945systemes} and was subsequently refined and corrected by Kuranishi \cite{kuranishi1957}.
\end{remark}

For simplicity assume $K$ is a field of constants, that is $\partial(K) = 0$. Although little can be said about the dimension of prolongation varieties in full generality, an exact value can be obtained under rather restrictive regularity assumptions.

{\begin{proposition}
    Let $\Sigma:= \{f_1(\mathbf{x}^{(\leq k)})=0,\ldots,f_r(\mathbf{x}^{(\leq k)})=0\}$ be a system of differential equations, each of order $k$, such that $Z(\Sigma) \subseteq \mathbb{A}^{n \times (k+1)}$ is irreducible of dimension $d$ and the following conditions hold:
    \begin{itemize}
        \item The projection $\vartheta:Z(\operatorname{Prol}(\Sigma)) \rightarrow Z(\Sigma)$ is surjective.
        \item The matrix $\left( \displaystyle \frac{\partial f_i}{\partial x_j^{(k)}} \right)_{i,j}$ has constant rank $\rho$ on $Z(\Sigma)$.
    \end{itemize}
    Then $Z(\operatorname{Prol}(\Sigma))$ is equidimensional of dimension $d+n-\rho$. Moreover, if $Z(\Sigma)$ is smooth, then $Z(\operatorname{Prol}(\Sigma))$ is smooth and irreducible.
\end{proposition}}
\begin{proof}
    See \cite[Proposition 1.4.9(i)]{Yudilevich2016}
\end{proof}

\begin{remark}\label{ejemplo prolongado}
For an ODE system $\Sigma:=\{\mathbf{x}'=f(\mathbf{x})\}$, the hypotheses of the preceding proposition hold automatically. Indeed, the Jacobian matrix with respect to the highest-order variables is the identity and hence has constant rank $n$. Consequently, every prolongation variety is isomorphic to the preceding one; in particular, all prolongation varieties have the same dimension.
\end{remark}

\subsection{Degree bounds} \ \medskip

As mentioned above, determining the dimension, degree, irreducibility, or even the number of irreducible components of a general prolongation variety remains largely open. In this subsection we combine the results of Subsection \ref{deg estiamtes} to obtain information about these invariants. In particular, we derive criteria guaranteeing the expected dimension and other geometric properties under successive prolongations.

\begin{notation}
    Let $r,k \in \mathbb{N}$ and let $\Sigma:= \{f_1(\mathbf{x}^{(\leq k)})=0,\ldots,f_r(\mathbf{x}^{(\leq k)})=0\}$ be a system of differential equations such that $Z(\Sigma)$ is smooth, irreducible and the ideal generated by the equations of $\Sigma$ is prime in $K[\mathbf{x}^{(\leq k)}]$. We use $T\Sigma$ to denote the tangent bundle of $Z(\Sigma)$, viewed as an algebraic variety in $\mathbb{A}^{2(n\times(k+1))}$.
\end{notation}

The intersection $T\Sigma \cap \{\mathbf{x}'=\mathbf{y}_1,\ldots,\mathbf{x}^{(k)}=\mathbf{y}_k\}$ is linearly isomorphic to $Z(\operatorname{Prol}(\Sigma))$. Bézout's inequality (see Proposition \ref{prop:bezout}) therefore gives
$$\deg(Z(\operatorname{Prol}(\Sigma))) \leq \deg(T\Sigma).$$

Applying Theorem \ref{teorema cuadrado}, we obtain the following result.

\begin{proposition}\label{prop:cota}
    Let $r,k \in \mathbb{N}$ and let $\Sigma:= \{f_1(\mathbf{x}^{(\leq k)})=0,\ldots,f_r(\mathbf{x}^{(\leq k)})=0\}$ be a system of differential equations such that $Z(\Sigma)$ is smooth, irreducible, and the ideal generated by the equations of $\Sigma$ is prime in
$K[\mathbf{x}^{(\leq k)}]$. Then
    $$\deg(Z(\operatorname{Prol}(\Sigma))) \leq  \deg(Z(\Sigma))^2.$$
\end{proposition}

These estimates apply to general dynamical systems. Their generality comes at the cost of sharpness, however, as the following example shows.

\begin{example}
    Consider
    $\Sigma:= \{
        x_1'=x_2^2, \ x_2'=x_1^{2}
    \}.$ It is immediate that $Z(\Sigma)$ is smooth and irreducible, with $\dim(Z(\Sigma))=2$ and $\deg(Z(\Sigma))=4$. Proposition~\ref{prop:cota} therefore gives
    $$\deg(Z(\operatorname{Prol}(\Sigma))) \leq \deg(Z(\Sigma))^2 = 16.$$
    On the other hand, standard computations show that $\deg(Z(\operatorname{Prol}(\Sigma)))=7$. Since this is a parametric setting, \cite[Theorem 25]{jeronimo2025geometric} shows that one should expect bounds linear in the degree rather than the quadratic bound obtained here.
\end{example}

\subsection*{Acknowledgements}  
This work was partially supported by Universidad de Buenos Aires (UBACYT 20020190100116BA), Argentina. The author is extremely grateful to his advisor and mentor Pablo Solernó for many stimulating discussions and valuable insights.

\bibliographystyle{abbrv}
\bibliography{bibliografiaTesis}{}
\end{document}